\documentclass[12pt]{amsart}
\usepackage{amssymb}
\allowdisplaybreaks

\usepackage{color}
\usepackage{graphics}
\usepackage{amsmath,amscd}
\usepackage{amsthm}
\usepackage[all,cmtip]{xy}
\usepackage{tikz}
\usepackage{graphicx}
\usepackage{epstopdf}
\usepackage{mathtools}
\usepackage{amssymb}
\usepackage{amsbsy}
\usepackage{bm}
\usepackage{tikz-cd}
\usepackage{enumerate}
\usepackage{enumitem}
\usepackage{float}
\usepackage{epsfig}
\theoremstyle{definition}

\newtheorem{exmp}{Example}[section]
\usepackage[colorlinks=green,linkcolor=blue,citecolor=red,urlcolor=red]{hyperref}

\newcommand{\SC}{{\mathcal{C}}}

\newcommand{\SV}{{\mathcal{V}}}

\newcommand{\Z}{\mathbb{Z}}

\newcommand{\D}{\mathbb{D}}
\newcommand{\CP}{\mathbb{CP}}
\newcommand{\s}{\mathbb{S}}

\newtheorem{proposition}{Proposition}[section]
\newtheorem{theorem}[proposition]{Theorem}
\newtheorem{definition}[proposition]{Definition}
\newtheorem{lemma}[proposition]{Lemma}

\newtheorem{corollary}[proposition]{Corollary}
\newtheorem{remark}[proposition]{Remark}
\newtheorem{question}[proposition]{Question}

\begin{document}
	
	\title{Surfaces in $4$-manifolds and extendible mapping classes }
	
	

	\author{Shital Lawande}
	\address{IAI TCG CREST Kolkata, and Ramkrishna Mission Vivekananda Education and Research Institute, Belur Math}
	\email{shital.lawande@tcgcrest.org}
	
	\author{Kuldeep Saha}
	\address{IAI TCG CREST Kolkata, and Academy of Scientific and Innovative Research, Ghaziabad}
	\email{kuldeep.saha@gmail.com, kuldeep.saha@tcgcrest.org}

	\begin{abstract}
		
		We study smooth proper embeddings of compact orientable surfaces in compact orientable $4$-manifolds and elements in the mapping class group of that surface which are induced by diffeomorphisms of the ambient $4$-manifolds. We call such mapping classes \emph{extendible}. An embedding for which all mapping classes are extendible is called \emph{flexible}. We show that for most of the surfaces there exists no flexible embedding in a $4$-manifold with homology type of a $4$-ball or of a $4$-sphere. As an application of our method, we address a question of Etnyre and Lekili and show that there exists no simple open book decomposition of $\s^5$ with a spin page where all $3$-dimensional open books admit open book embeddings. We also provide many constructions and criteria for extendible and non-extendible mapping classes, and discuss a connection between extendibility and sliceness of links in a homology $4$-ball with $\s^3$ boundary. Finally, we give a new generating set of the group of extendible mapping classes for the trivial embedding of a closed genus $g$ surface in $\s^4$, consisting of $3g$ generators. This improves a previous result of Hirose giving a generating set of size $6g-1$.   
		
	\end{abstract}
	
	\maketitle
	
	\section{Introduction}
	
	We study smooth embeddings of orientable surfaces in $4$-manifolds and elements in the mapping class group of that surface which are induced by diffeomorphisms of the ambient $4$-manifolds. We call such mapping classes \emph{extendible} with respect to the given embedding. Embeddings for which the whole mapping class group is extendible, are called \emph{flexible}. The problem of extendibility originated in the work of Montesinos \cite{Mont}, who considered the case of an unknotted torus in $\s^4$ to construct a diffeomorphism of $\s^4$ known as the  \emph{Montesinos twist}. The general question of extendibility/flexibility for closed oriented surfaces was further studied by Hirose \cite{hirose0} \cite{hirose1} \cite{hirose2} and Hirose--Yasuhara \cite{hy}. Later, a similar problem for surfaces with boundary came up in the context of codimension $2$ \emph{open book embedding}(or \emph{spun embedding}) of smooth and contact $3$-manifolds.
	
	\noindent An open book decomposition of a manifold $M^m$ is a pair $(V^{m-1},h)$, such that $M^m$ is diffeomorphic to the quotient space $\mathcal{MT}(V^{m-1}, h) \cup_{id} \partial V^{m-1} \times D^2$. Here, $V^{m-1}$, called \emph{page}, is a manifold with boundary. The map $h$, called \emph{monodromy}, is a diffeomorphism of $V^{m-1}$ that restricts to identity near the boundary $\partial V$, and $\mathcal{MT}(V^{m-1}, h)$ denotes the mapping torus of $h$. We denote such an open book
	by $\textrm{OB}(V,h)$. It is well known that every closed, orientable, odd dimensional manifold admits open book decomposition. An embedding $f$ of $\textrm{OB}(V_1,h_1)$ in $\textrm{OB}(V_2,h_2)$ is called an \emph{open book embedding}, if $f$ restricts to a proper embedding of $V_1$ in $V_2$ and $h_2 \circ f = f \circ h_1$ (up to isotopy).  
	
	\noindent Let $\Sigma_{g,b}$ denote a genus $g$ surface with $b$ boundary components and let $\tilde{f}$ be a proper embedding of $\Sigma_{g,b}$ in a compact $4$-manifold $(V^4, \partial V^4)$. If $\phi$ is a diffeomorphism of $\Sigma_{g,b}$ that can be induced via a boundary preserving isotopy of $V^4$, then $\textrm{OB}(\Sigma_{g,b}, \phi)$ open book embeds in $\textrm{OB}(V^4,id.)$. Thus, $\phi$ being extendible with respect to $\tilde{f}$ is a \emph{necessary} condition for the existence of such an open book embedding. The study of codimension $2$ open book embedding of closed $3$-manifolds has received much attention in recent times. In particular, many codimension $2$ open book embeddings have been constructed for $3$-manifolds, by the works of Etnyre--Lekili \cite{EL}, Pancholi--Pandit--Saha \cite{pps} and Saha \cite{Saha}. It is known that there are $5$-dimensional open books where every open book decomposition of a $3$-manifold admits an open book embedding. However, for topologically simpler $5$-manifolds (eg. integral homology $5$-sphere) such results are still unknown. Etnyre and Lekili have asked the following question.

	\begin{question}[Qn. $3.5$ in \cite{EL}]\label{qn1}
		Is there an open book decomposition of $S^5$ where all open book decompositions of a closed orientable $3$-manifold embed?
	\end{question}
	
	\noindent We call such an open book decomposition a \emph{universal open book} of $\s^5$. In particular, one may ask whether the trivial open book on $\s^5$, $\textrm{OB}(\D^4,id)$, is an universal open book. In terms of extendibility of mapping classes, this is equivalent to the following question.
	
	\begin{question}\label{qn2}
		Does there exist a flexible proper embedding of an orientable bounded surface in $(\D^4,\partial \D^4)$?
	\end{question}
	
	We provide a complete answer to question \ref{qn2} by showing the following.

	\begin{theorem} \label{thm1}
		
		 Let $W^4$ be an orientable compact $4$-manifold with boundary (possibly empty). Assume that either $W$ is an integral homology $4$-ball with boundary an integral homology $3$-sphere, or $W$ is an integral homology $4$-sphere. For $g \geq 0$ and $b \geq 0$, let $\Sigma_{g,b}$ denote a compact orientable surface of genus $g$ with $b$ boundary components. 
		
		\begin{enumerate}
			\item For $g \ge 1$ and $b \neq 1$, there exists no proper flexible embedding of $(\Sigma_{g,b}, \partial \Sigma_{g,b})$ in $(W,\partial W )$.
			
			\item For $b \geq 3$ there exists no proper flexible embedding of $(\Sigma_{0,b}, \partial \Sigma_{0,b})$ in $(W,\partial W )$.
		\end{enumerate}

	\end{theorem}

	\noindent In fact, only $\Sigma_{1,1}$ and $\Sigma_{0,2}$ admit proper flexible embedding in $\D^4$. Examples of such embeddings are given in section \ref{sec4}. See Proposition \ref{pr1} and example \ref{flex1}. The cases of $\s^2$ and $\D^2$ are trivial.
	
	\vspace{0.1cm}
	
	We call an open book decomposition of $\s^5$ \emph{simple} if it has a connected binding and a simply connected page. As an application of our method to prove Theorem \ref{thm1}, we answer question \ref{qn1} for all simple open book decompositions of $\s^5$ with a spin page.

	\begin{theorem}\label{thm1.5}
		
		There exists no universal simple open book decomposition of $\s^5$ with a spin page.	
		
	\end{theorem}
	
	\noindent The open books $\textrm{OB}(\D^4,id)$, $\textrm{OB}(DT^*S^2, \tau_2)$ and $\textrm{OB}((\s^2 \times \s^2  \#  \s^2 \times \s^2) \setminus \D^4, h_0)$ of $\s^5$ appear often in the study of low dimensional manifolds and contact topology. Here, $DT^*S^2$ is diffeomorphic to the unit $\D^2$-bundle over $\s^2$ with euler number $-2$ and $\tau_2$ denotes the Dehn-Seidel twist. For more details on this open book we refer to \cite{vanko} and \cite{Saha}. For the open book $\textrm{OB}(\s^2 \times \s^2 \# \s^2 \times \s^2 \setminus \D^4, h_0)$ we refer to \cite{Sk1} and \cite{Sk2}. By Theorem \ref{thm1.5}, none of these three open books are universal. In fact, the following two theorems of Saeki show the vastness of examples for simple open book decompositions of $\s^5$ with spin pages.
    
    \begin{theorem}(Theorem $6.1$ in \cite{Sk1}) Let $M$ be a closed orientable connected $3$-manifold. There exists a simple open book decomposition of $\s^5$ with a spin page bounding $M$. 
    
    \end{theorem}

    \begin{theorem}(Theorem $4.1$ in \cite{Sk2}) Let $V$ be a smooth compact simply connected $4$-manifold with nonempty connected boundary. If $V$ is spin, then there exists a simple open book decomposition of $\s^5$ with page diffeomorphic to $V \# k(\s^2 \times \s^2)$ for some nonnegative integer $k$.
    	
    \end{theorem}
    
     \noindent Here, $V \# k(\s^2 \times \s^2)$ denotes the $k$-fold connected sum of $V$ with $\s^2 \times \s^2$. Since $\s^2 \times \s^2$ is spin, $V \# k(\s^2 \times \s^2)$ is spin too.
     
     \vspace{0.25cm}

	In the remaining of this section, we discuss the organisation of this article and the rest of our results related to extendible mapping classes and their applications. 
	
	\vspace{0.1cm}
	
	\noindent In section \ref{sec2}, we discuss some preliminary notions and results about surface embeddings, extendiblity of mapping classes, the Rokhlin quadratic form and open books. Section \ref{sec3} then discusses the proofs of Theorem \ref{thm1} and Theorem \ref{thm1.5}.
	
	\vspace{0.1cm}
	
	\noindent Continuing our study of extendible mapping classes, in section \ref{sec4}, we explore various examples and constructions of extendible and non-extendible mapping classes of embedded surfaces. In particular, we provide many such examples of properly embedded surfaces constructed from seifert surfaces of knots and links in $\s^3$. See Theorem \ref{thm2}, Theorem \ref{thm3} and Theorem \ref{thmham}.
	
	\vspace{0.1cm}
	
	\noindent In section \ref{sec5}, we discuss a relation between Theorem \ref{thm1} and sliceness (smooth) of links in homology $4$-balls with boundary $\s^3$. We show how Theorem \ref{thm1} can be used to obstruct sliceness in such cases. See Theorem \ref{thm4}. 
	
	\vspace{0.1cm}
	
	\noindent Finally, in section \ref{sec6}, we introduce a \emph{fibered Dehn twist type} diffeomorphism of $\s^4$, denoted by $\textrm{T}_{\s^3}$, using the flow of an open book decomposition on $\s^3$. We show that for the trivial embedding of a closed orintable genus $g$ surface in $\s^4$, the map $\textrm{T}_{\s^3}$ is enough (up to isotopy of $\s^4$) to induce all elements in the subgroup of extendible mapping classes. Moreover, using the description of $\textrm{T}_{\s^3}$, we give a generating set of the extendible mapping class group with $3g$ generators (see Theorem \ref{thm5}). This improves a previous result of Hirose \cite{hirose1} giving a generating set with $6g-1$ generators.

	\vspace{0.1cm}

	Unless stated otherwise, we always work in the category of smooth, connected and orientable manifolds. 
	
	\vspace{0.1cm}
		
	\subsection*{Acknowledgement} The authors thank Sukumar Das Adhikari for his support and encouragement during this work. We also thank Maggie Miller for a clarification regarding Rokhlin quadratic forms.

	\section{Preliminaries}	\label{sec2}
	
	\subsection{Open books and embeddings} An open book decomposition of a closed $(2n+1)$-manifold $M$ consists of a codimension $2$ closed submanifold $B$ and a fibration map $\pi : M \setminus B \rightarrow \s^1$, such that in a tubular neighborhood of $B \subset M$, the restriction map $\pi : B \times (\D^2 \setminus \{0\}) \rightarrow \s^1$ is given by $(b,r,\theta) \mapsto \theta$. The fibration $\pi$ determines a unique fiber manifold $N^{2n}$ whose boundary is $B$. The closure $\bar{N}$ is called the \emph{page} and $B$ is called the \emph{binding}. The monodromy of the fibration map $\pi$ determines a diffeomorphism $\phi$ of $\bar{N}$ such that $\phi$ is identity near the boundary $\partial \bar{N}$. In particular, $M = \mathcal{MT}(\bar{N}, \phi) \cup _{id, \partial} \partial \bar{N} \times \D^2$, where $\mathcal{MT}(\bar{N}, \phi)$ denotes the mapping torus of $\phi$. We denote such an open book decomposition of $M$ by $\textrm{OB}(\bar{N},\phi)$. The map $\phi$ is called the \emph{monodromy} of the open book. It is a well-known fact that every $M^{2n+1}$ admits an open book decomposition. 
	
	\begin{exmp} 
		The \emph{trivial} open book on $\s^n$ is given by $\textrm{OB}(\D^{n-1}, id)$. 
	\end{exmp} 
	
	\noindent Given an open book $\textrm{OB}(V^{2n},h)$, the \emph{variation map} of the monodromy $h$ is defined as the homomorphism $\Delta_h : H_n(V, \partial V; \mathbb{Z}) \rightarrow H_n(V; \mathbb{Z})$ induced by $(id_V)_* - h_*$.

	\begin{definition}[Simple open book]
		An open book $\textrm{OB}(V^{2n},h)$ ($n \geq 2$) is called simple if $V$ is $(n-1)$-connected and $\partial V$ is $(n-2)$-connected. 
	\end{definition}

	\begin{definition}[Open book embedding]
		
		We say, $\textrm{OB}(\Sigma_1,\phi_1)$ \emph{open book embeds} in $\textrm{OB}(\Sigma_2,\phi_2)$, if there exists a proper embedding $g : (\Sigma_1,\partial \Sigma_1) \rightarrow (\Sigma_2,\partial \Sigma_2)$ such that $g \circ \phi_1$ is isotopic to $\phi_2 \circ g$, relative to the boundary. 
		
	\end{definition}

	\subsection{Extendible mapping classes of embedded surfaces}
	
	Let $f : (\Sigma, \partial \Sigma) \rightarrow (V, \partial V)$ be a proper embedding of a compact orientable surface $\Sigma$ in a compact orientable $4$-manifold $V$. Let $\mathcal{MCG}(\Sigma, \partial \Sigma)$ denote the relative mapping class group of $\Sigma$. We denote the Dehn twist along a simple closed curve $\gamma \subset \Sigma$ by $\tau_\gamma$.
	
	\begin{definition}
   		A mapping class $\phi \in \mathcal{MCG}(\Sigma, \partial \Sigma)$ is called \emph{$f$-extendible} if there exists a diffeomorphism $H$ of $V$ (that restricts to identity near $\partial V$) such that $H \circ f = f \circ \phi$, up to isotopy.  
	\end{definition}
	
	\noindent It is not hard to see that extendibility is rather a property of a mapping class of the surface. In particular, we note the following.
	
	\begin{lemma} \label{Leic}
		Let $f$ be a proper embedding of a surface $F$ in a $4$-manifold $V$. Let $\alpha$ and $\beta$ be  two isotopic simple closed curves on $F$. If $\tau_\alpha$ is $f$-extendible, then $\tau_\beta$ is also $f$-extendible. 
	\end{lemma}

	\noindent Note that the set of all $f$-extendible mapping classes form a subgroup of $\mathcal{MCG}(\Sigma,\partial \Sigma)$. Moreover, extendibility is also preserved under isotopy of an embedding.
	
	\begin{lemma}\label{4.4}
	Let $f_0, f_1 : F \rightarrow M$ be proper embeddings such that  $f_0$ is isotopic to $f_1$ (relative to boundary), and let $\gamma$ be a simple closed curve on $F$. If $\tau_\gamma$ is $f_1$-extendible then $\tau_\gamma$ is $f_0$-extendible. 
	\end{lemma}
	
	\begin{proof}
	Given that the embeddings $f_0, f_1 : F \rightarrow M$ are isotopic, there exists an isotopy $h_t: F \rightarrow M$ ($t\in [0,1]$) such that $h_0=f_0$ and $h_1=f_1.$  We extend this to an isotopy $H_t : M \rightarrow M$ ($t \in [0,1]$) such that $H_0 = id_M$ and $H_1 \circ f_0 = f_1.$ Since $\tau_\gamma$ is $f_1$-extendible, there exists a diffeomorphism  $\Psi : M \rightarrow M$ such that $\Psi \circ f_1 = f_1 \circ \tau_\gamma.$ Therefore, $({H_1 }^{-1}\circ\Psi \circ H_1 )\circ f_0 =  f_0 \circ \tau_\gamma$ (up to isotopy). Hence, $\tau_\gamma$ is $f_0$-extendible. 
	\end{proof}
	
	\begin{definition}[Flexible embedding]
		A proper embedding $f$ of a surface $\Sigma$ in a $4$-manifold $V$ is called flexible if every element of $\mathcal{MCG}(\Sigma,\partial \Sigma)$ is $f$-extendible.
	\end{definition}
		
	\subsection{Characteristic embedding and Rokhlin quadratic form}
	
	Let $\Sigma$ be a compact orientable surface and let $V$ be a compact orientable $4$-manifold such that $H_1 (V; \Z_2) = 0$. A proper embedding $f : (\Sigma, \partial \Sigma) \rightarrow (V, \partial V)$ is called \emph{chraracteristic} if $[f(\Sigma)] \cdot X \equiv X\cdot X \pmod 2$  for all $X \in H_2(M^4, \mathbb{Z}_2)$. Equivalently, $[f(\Sigma)] \in H_2(V, \partial V; \mathbb{Z}_2)$ is the Poincare dual of the second Stiefel–Whitney class of $V$. Given such a characteristic embedding $f$, one can define a quadratic form $q_f:
	H_1 (F,\mathbb{Z}_2) \rightarrow \Z_2$ in the following way.  
	
	\noindent For $x \in H_1(\Sigma, \mathbb{Z}_2)$ we choose a simple closed curve $C$ on $\Sigma$ representing $x$. As $H_1(V; \mathbb{Z}_2)=0$, $f(C)$ bounds a connected, orientable surface $D$ embedded in $V^4$. We may assume that $D$ meets $\Sigma$ transversely. Since $D$ is homotopy equivalent to an wedge of circles, the normal bundle ${\nu}_D$ of $D$ is trivial. A trivialization on $D$ induces a unique trivialization on the restriction of $\nu_D$ over $\partial D = f(C)$ (see section $2$ in \cite{kl}). The normal bundle of $f(C)$ in $f(\Sigma)$ then determines a $1$-dimensional subbundle of this trivialized $2$-disk bundle over $C$. Let $\mathcal{O}(D)$ be the mod $2$ number of full twists made by this $1$-dimensional subbundle as we go around $C$. Now define, $$ q_f(x)= D \cdot f(\Sigma) + \mathcal{O}(D)\mathrm{mod~2},$$ where $D \cdot f(\Sigma)$ is the number of intersection points between the interior of $(D)$ and $f(\Sigma)$. The quadratic form $q_f$ is called the \emph{Rokhlin quadratic form} of the embedding $f$. For $x, y \in H_1(\Sigma; \mathbb{Z}_2)$, $q_f$ satisfies the relation: $q_f(x + y) = q_f(x) + q_f(y) + x \cdot y$, where $x \cdot y$ is the $\pmod2$ intersection number of $x$ and $y$. 
	
    \noindent Recall that $\phi \in \mathcal{MCG}(\Sigma, \partial \Sigma)$ is $f$-extendible, if there is an orientation preserving diﬀeomorphism $\Phi$ of $V$ (up to isotopy) such that $f \circ \phi = \Phi \circ f$. If $\phi$ is $f$-extendible then, by definition of the Rokhlin quadratic form, $q_f([\phi(\alpha)]) = q_f([\alpha])$ for every simple closed curve $\alpha$ on $\Sigma$ (see remark $3.2$ in \cite{hy}).

	\subsection{Rokhlin quadratic form for some surfaces in $\mathbb{D}^4$} \label{ssec2}
	
	Let $L \subset \mathbb{S}^3 $ be a link, and let $F$ be a Seifert surface of $L$ diffeomorphic to $\Sigma_{g,b}$. Consider a collar neighborhood $\s^3 \times [0,1]$ of $\partial \mathbb{D}^4 = \mathbb{S}^3 $ in $ \mathbb{D}^4$ such that $\s^3 \times \{1\} = \partial \D^4$. We construct a proper embedding $f$ of $\Sigma_{g,b}$, by isotoping the interior of $F$ into the interior of the collar (fixing the boundary link $L$ pointwise) such that $f(\Sigma_{g,b}) \cap \mathbb{S}^3 \times \{0\}$ is isotopic to $F$ in $\s^3 \times \{0\}$, and $f(\Sigma_{g,b}) \cap \s^3 \times \{t\} = L \times \{t\}$ for $t \in (0,1]$. We call $f$ is a \emph{Seifert type embedding} of $\Sigma_{g,b}$ corresponding to the Seifert surface $F$. It is easy to see that Seifert type embeddings can be defined for any $4$-manifold with boundary $\s^3$. Moreover, we can isotope the embedding (relative to boundary) such that $f(\Sigma_{g,b}) \cap \mathbb{S}^3 \times \{c\} = F$ for any $c \in [0,1)]$, $f(\Sigma_{g,b}) \cap \mathbb{S}^3 \times [0,c)] = \emptyset$, and $f(\Sigma_{g,b}) \cap \s^3 \times \{t\} = L \times \{t\}$ for $t \in (c,1].$
	
	\vspace{0.1cm}
	
	\noindent Given a knot $K \subset \s^3$ with framing $n \in \mathbb{Z}$, we get a $2$-component link $L_K$ with linking number $n$ and a ribbon surface bounding $L_K$. Let $\mathcal{R}(K,n)$ denote the corresponding Seifert type embedding of an annulus in $\D^4$. 
	
	\noindent Let $F$ be a Seifert surface of $K$ and let $\gamma \subset F$ be an unknot. Let $f$ be a Seifert type embedding corresponding $F$. Let $D_2 = \gamma \times [0,\frac{1}{2}] \cup D_1$, where $D_1$ is a $2$-disk in $\s^3 \times \{0\}$ and $\partial D_2 = \gamma \times \{0\}$. Then, $D_2$ is a $2$-disk such that $int(D_2) \cdot Im(f) = 0$. Since $\gamma$ bounds a two disk isotopic to the standard $2$-disk in $\s^3 \times \{\frac{1}{2}\}$, the trivialization of the normal bundle $\nu_{D_2}$ over $\gamma$ is the zero trivialization. Therefore, by definition of the Rokhlin quadratic form, $q_f([\gamma]) = \mathcal{O}(D_2)$ is given by the mod $2$ surface framing of $\gamma$ on $F.$ In particular, if a tubular neighborhood of $\gamma$ on $F$ is isotopic to $\mathcal{R}(K,n)$, then $q_f([\gamma]) = n \pmod 2$.

	\begin{exmp}[Quadratic form of trivially embedded $\Sigma_{g,1}$ inside $\mathbb{D}^4$.]
		
		We embed $\Sigma_{g,1}$ in $\mathbb{S}^3\times \{\frac{1}{2}\} \subset \s^3 \times [0,1]$ by removing a $2$-disk from the boundary of a standard genus $g$ handlebody embedded in $\mathbb{S}^3\times \{\frac{1}{2}\}$. Let $C$ denote the boundary of the embedded surface. We attach the cylinder $C \times [\frac{1}{2},1]$ to get a proper embedding $f$ of $\Sigma_{g,1}$, as shown in figure \ref{sig71}. Note that $f$ is a Seifert type embedding corresponding to the standard genus $g$ Seifert surface of the unknot.
		
		\begin{figure}[!htb]
			\centering
			\def\svgwidth{11cm}
\begingroup%
  \makeatletter%
  \providecommand\color[2][]{%
    \errmessage{(Inkscape) Color is used for the text in Inkscape, but the package 'color.sty' is not loaded}%
    \renewcommand\color[2][]{}%
  }%
  \providecommand\transparent[1]{%
    \errmessage{(Inkscape) Transparency is used (non-zero) for the text in Inkscape, but the package 'transparent.sty' is not loaded}%
    \renewcommand\transparent[1]{}%
  }%
  \providecommand\rotatebox[2]{#2}%
  \newcommand*\fsize{\dimexpr\f@size pt\relax}%
  \newcommand*\lineheight[1]{\fontsize{\fsize}{#1\fsize}\selectfont}%
  \ifx\svgwidth\undefined%
    \setlength{\unitlength}{483.3456174bp}%
    \ifx\svgscale\undefined%
      \relax%
    \else%
      \setlength{\unitlength}{\unitlength * \real{\svgscale}}%
    \fi%
  \else%
    \setlength{\unitlength}{\svgwidth}%
  \fi%
  \global\let\svgwidth\undefined%
  \global\let\svgscale\undefined%
  \makeatother%
  \begin{picture}(1,0.6365465)%
    \lineheight{1}%
    \setlength\tabcolsep{0pt}%
    \put(0,0){\includegraphics[width=\unitlength,page=1]{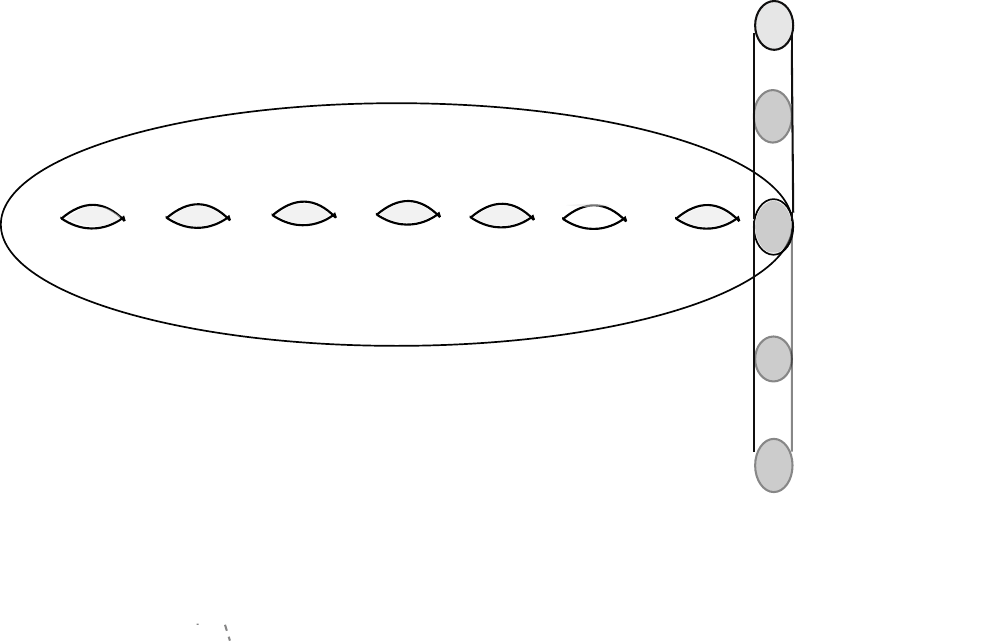}}%
    \put(0.81653412,0.61496285){\color[rgb]{0,0,0}\makebox(0,0)[lt]{\lineheight{1.25}\smash{\begin{tabular}[t]{l}$\mathbb{S}^3\times 1$\end{tabular}}}}%
    \put(0.81748659,0.52733445){\color[rgb]{0,0,0}\makebox(0,0)[lt]{\lineheight{1.25}\smash{\begin{tabular}[t]{l}$\mathbb{S}^3\times 3/4$\end{tabular}}}}%
    \put(0.81367671,0.42160893){\color[rgb]{0,0,0}\makebox(0,0)[lt]{\lineheight{1.25}\smash{\begin{tabular}[t]{l}$\mathbb{S}^3\times 1/2$\end{tabular}}}}%
    \put(0.81081929,0.28254653){\color[rgb]{0,0,0}\makebox(0,0)[lt]{\lineheight{1.25}\smash{\begin{tabular}[t]{l}$\mathbb{S}^3\times 1/4$\end{tabular}}}}%
    \put(0.84034619,0.17015366){\color[rgb]{0,0,0}\makebox(0,0)[lt]{\lineheight{1.25}\smash{\begin{tabular}[t]{l}$\mathbb{S}^3\times 0$\end{tabular}}}}%
    \put(0.75652775,0.39112951){\color[rgb]{0,0,0}\makebox(0,0)[lt]{\lineheight{1.25}\smash{\begin{tabular}[t]{l}$\gamma$\end{tabular}}}}%
    \put(0,0){\includegraphics[width=\unitlength,page=2]{Proper_embedding_of_F.pdf}}%
  \end{picture}%
\endgroup%

			\vspace{-1.5cm}
			\caption{Standardly embedded $\Sigma_{7,1}$ inside $\mathbb{D}^4$ }
			\label{sig71}
		\end{figure}

		Since $H_2(\mathbb{D}^4 ; \mathbb{Z})=0$, $f$ is characteristic. Let $c \subset \Sigma_{g,1}$ be an unknot. As seen before, we can find a 2-disk $D$ in $\mathbb{D}^4$ (transverse to $f$) such that $\partial D = c$, $int(D) \cdot Im(f) = 0$, and $D$ is isotopic to the standard $2$-disk in $\s^3$. Therefore, $q_f([c])$ is given by the surface framing of $f(c)$ on $f(\Sigma_{g,1})$. Since $\Sigma_{g,1}$ is trivially embedded in $\D^4$, $q_f$ of a meridinal curve $m$ and a longitudinal curve $l$ is zero. If $\gamma$ is a curve on $\Sigma_{g,1}$ such that $[\gamma] = [m] + [l]$ and $m \cdot l = 1$, then $q_f([\gamma]) = 1.$ 
	\end{exmp}

	\subsection{Generalized Dehn twists along $3$-manifolds} Let $Y$ be a closed oriented $3$-manifold, and $\rho : \s^1 \rightarrow Diff^+(Y )$ be a loop based at the identity map. Given any $\s^1$-action on $Y$ one can find such a loop. We can define a diffeomorphism $ \textrm{T}_Y : Y \times [0, 1] \rightarrow Y \times [0, 1]$ by $(y, t) \mapsto (\rho(t)(y), t)$ for $(y, t) \in Y \times [0, 1]$. Here, we think of $\s^1$ as the quoitent $[0,1]/\{0\sim1\}$. 
	
	\noindent Whenever there is an embedding of $Y$ into a $4$-manifold $V^4$, we can extend $\textrm{T}_Y$ to a diffeomorphism of $V^4$ supported in a neighborhood of the embedded $Y$. We call $\textrm{T}_Y$ a  \emph{fibered Dehn twist} on $V^4$ along $Y$. Although $\textrm{T}_Y$ depends on $\rho$, whenever $\rho$ is clear from the context, we shall use $\textrm{T}_Y$ to denote this Dehn twist.

	\section{Existence problem of flexible codimension 2 embedding of surface}\label{sec3}

	Let $W$ be a compact orientable $4$-manifold with $H_1(W; \mathbb{Z}) = 0$. Our main ingredient to prove Theorem \ref{thm1} and Theorem \ref{thm1.5} is the following lemma. As before, $\Sigma_{g,b}$ will denote an orientable surface of genus $g$ with $b$ boundary components.
	
	\begin{lemma} \label{mainlemma}

    Let $f : (\Sigma_{g,b}, \partial \Sigma_{g,b}) \rightarrow (W^4, \partial W^4)$ be a characteristic proper embedding.
    
    \begin{enumerate}
    	
    	\item[(a)] Let $\gamma_1$ be an essential simple closed curve in the interior of $\Sigma_{g,b}$. If $\tau_{\gamma_1}$ is $f$-extendible then $q_f([\gamma]) = 1$.
    	
    	\item[(b)] Assume that $\partial W$ is an integral homology sphere. Let $\gamma_2$ be a non-nullhomologous simple closed curve in the interior of $\Sigma_{g,b}$ which is parallel to a boundary component of $\Sigma_{g,b}$. If $\tau_{\gamma}$ is $f$-extendible, then $q_f([\gamma_2]) = 1$.  
    	
    	\item[(c)] If $(g,b) \notin \{(1,1), (0,2),(1,0)\}$, then there exists a simple closed curve $\gamma$ in the interior of $\Sigma_{g,b}$ such that $q_f(\gamma) = 0$ and $\gamma$ is either essential or parallel to a boundary component.

    \end{enumerate}

    \end{lemma}

	\begin{proof}
		
			Recall that for $x,y \in H_1(\Sigma_{g,b};\mathbb{Z}_2)$, $q_f(x+y) = q_f(x) + q_f(y) + x \cdot y$.
		
		\begin{enumerate}

			\item[(a)] Since $\gamma_1$ is essential, there exists another simple closed curve $\tilde{\gamma}_1$ in the interior of $\Sigma_{g,b}$ such that $\gamma_1$ intersects $\tilde{\gamma_1}$ at a single point. Thus, $q_f([\tau_{\gamma_1}(\tilde{\gamma}_1)]) = q_f([\gamma_1] + [\tilde{\gamma}_1]) = q_f([\gamma_1]) + q_f([\tilde{\gamma}_1]) +1$.  If $\tau_{\gamma_1}$ is $f$-extendable, then $q_f([\tau_{\gamma_1}(\tilde{\gamma}_1)]) = q_f([\tilde{\gamma}_1])$. Therefore, $q_f([\gamma_1]) = 1$.

			\item[(b)] Let $DW$ denote the double of $W$, obtained by gluing two copies of $W$ along their boundaries. We can assume (by smoothing corners) that $DW$ is smooth. Since $\partial W$ is an integral homology sphere, $H_1(DW;\mathbb{Z}) = 0$. Moreover, under doubling, $f$ induces a characteristic embedding, $f'$, of the double of $\Sigma_{g,b}$ (a closed surface of genus $2g+b-1$) into $DW$. Figure \ref{fig-106} demonstrates the double for $\Sigma_{0,5}$. The curve $(\gamma_2)$ then becomes essential in the double of $\Sigma_{g,b}$. Since $\tau_{\gamma_2}$ is $f$-extendible, $\tau_{\gamma_2}$ is also $f'$-extendible. By the previous case of $(a)$, we get $q_{f'}([\gamma_2]) = 1$. Since $f$ and $f'$ coincide in a neighborhood of $\gamma_2$, $q_f([\gamma_2])=q_{f'}([\gamma_2])=1$. 
			
			\begin{figure}[!htb]  
				\centering  
				\def\svgwidth{8cm}  
\begingroup%
  \makeatletter%
  \providecommand\color[2][]{%
    \errmessage{(Inkscape) Color is used for the text in Inkscape, but the package 'color.sty' is not loaded}%
    \renewcommand\color[2][]{}%
  }%
  \providecommand\transparent[1]{%
    \errmessage{(Inkscape) Transparency is used (non-zero) for the text in Inkscape, but the package 'transparent.sty' is not loaded}%
    \renewcommand\transparent[1]{}%
  }%
  \providecommand\rotatebox[2]{#2}%
  \newcommand*\fsize{\dimexpr\f@size pt\relax}%
  \newcommand*\lineheight[1]{\fontsize{\fsize}{#1\fsize}\selectfont}%
  \ifx\svgwidth\undefined%
    \setlength{\unitlength}{637.74931394bp}%
    \ifx\svgscale\undefined%
      \relax%
    \else%
      \setlength{\unitlength}{\unitlength * \real{\svgscale}}%
    \fi%
  \else%
    \setlength{\unitlength}{\svgwidth}%
  \fi%
  \global\let\svgwidth\undefined%
  \global\let\svgscale\undefined%
  \makeatother%
  \begin{picture}(1,0.82620724)%
    \lineheight{1}%
    \setlength\tabcolsep{0pt}%
    \put(0,0){\includegraphics[width=\unitlength,page=1]{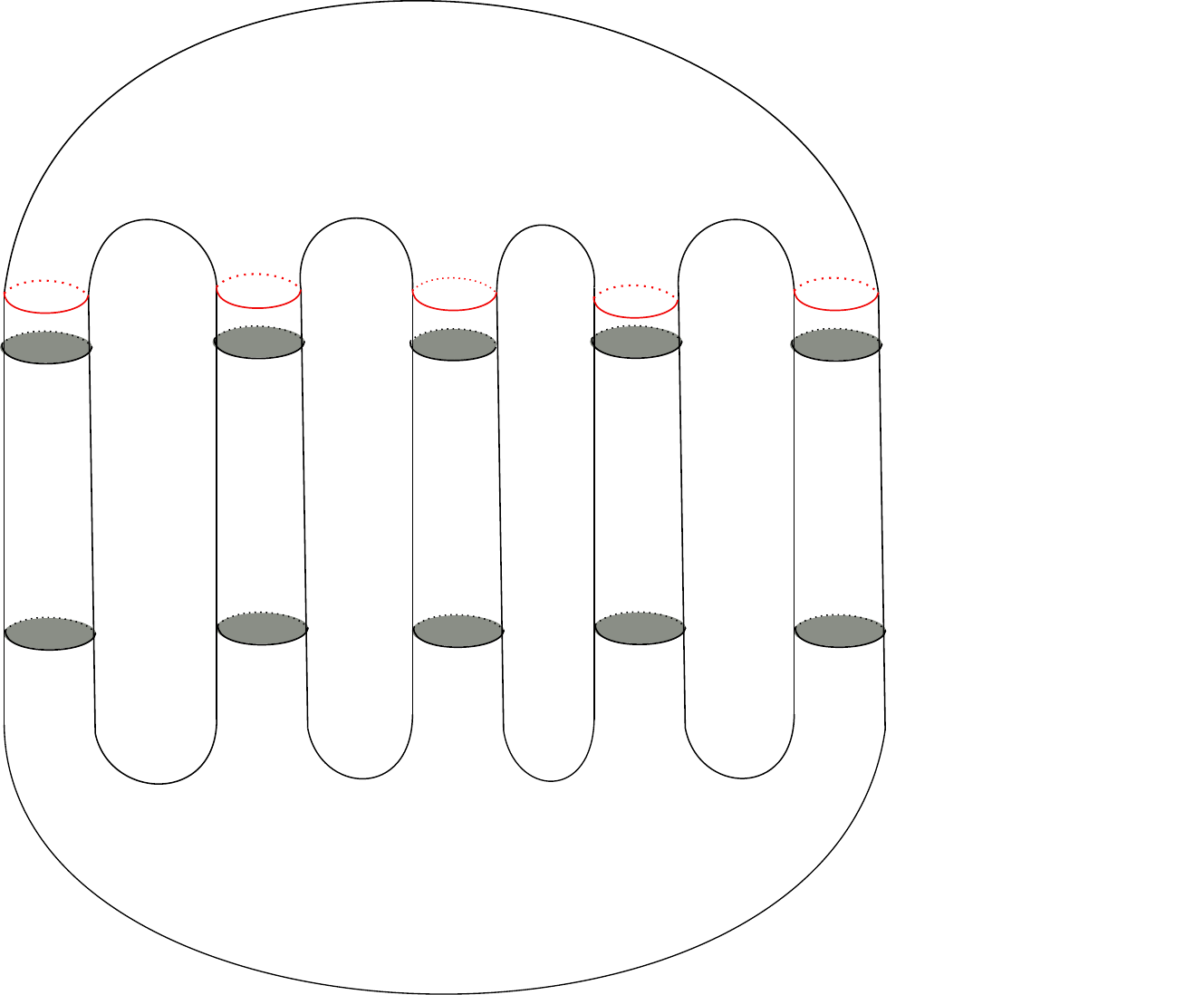}}%
    \put(0.72982074,0.58340785){\color[rgb]{0,0,0}\makebox(0,0)[lt]{\lineheight{1.25}\smash{\begin{tabular}[t]{l}$\Sigma_{0,5} \times \{0\}$\end{tabular}}}}%
    \put(0.72559436,0.24103194){\color[rgb]{0,0,0}\makebox(0,0)[lt]{\lineheight{1.25}\smash{\begin{tabular}[t]{l}$\Sigma_{0,5} \times \{1\}$\end{tabular}}}}%
  \end{picture}%
\endgroup%
  
				\caption{}  
				\label{fig-106}  
			\end{figure}
			
			\item[(c)] 	There are three possible cases.
			
			\textbf{Case 1} ($g \geq 2$ and $b \geq 0$) : Let us take a set of simple closed curves $\{a_1,a_2,..., a_g, c_1, \linebreak c_2,..., c_{g-1}, b_1, b_2\}$ on $\Sigma_{g,b}$, as in figure \ref{fig-15}. Consider the disjoint essential curves $b_1, c_1, b_2$. Note that $[b_1] = [c_1] + [b_2] \in H_2(\Sigma_{g,b};\mathbb{Z}_2)$. Thus, $q_f([b_1]) = q_f([c_1]) + q_f([b_2])$ and therefore at least one of $b_1, c_1, b_2$ has $q_f = 0$.

	    	\begin{figure}[!htb]
				\centering
				\def\svgwidth{12cm}
				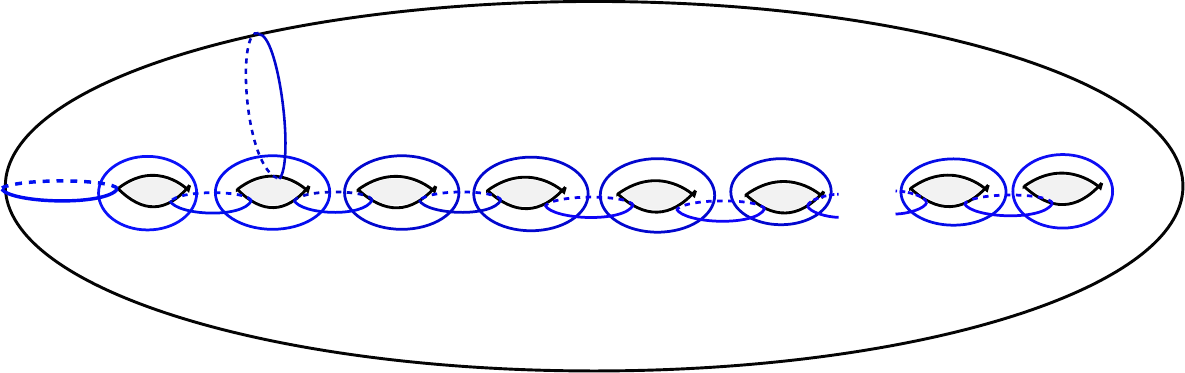
				\caption{}
				\label{fig-15}
			\end{figure}
			
			\textbf{Case 2} ($g = 0$ and $b \geq 3$) : Consider the disjoint simple closed curves $\gamma_1$, $\gamma_2$, and $\gamma$ on $\Sigma_{0,b}$, as in figure \ref{fig-13}. Since $[\gamma]=[\gamma_1] + \gamma_2] \in H_2(\Sigma_{0,b}; \mathbb{Z}_2)$, we have $q_f([\gamma]) = q_f([\gamma_1]) + q_f([\gamma_2])$. Therefore, $q_f([\alpha]) = 0$ for some $\alpha \in \{\gamma, \gamma_1, \gamma_2\}$.  
			
			\begin{figure}[!htb]  
				\centering  
				\def\svgwidth{8cm}  
\begingroup%
  \makeatletter%
  \providecommand\color[2][]{%
    \errmessage{(Inkscape) Color is used for the text in Inkscape, but the package 'color.sty' is not loaded}%
    \renewcommand\color[2][]{}%
  }%
  \providecommand\transparent[1]{%
    \errmessage{(Inkscape) Transparency is used (non-zero) for the text in Inkscape, but the package 'transparent.sty' is not loaded}%
    \renewcommand\transparent[1]{}%
  }%
  \providecommand\rotatebox[2]{#2}%
  \newcommand*\fsize{\dimexpr\f@size pt\relax}%
  \newcommand*\lineheight[1]{\fontsize{\fsize}{#1\fsize}\selectfont}%
  \ifx\svgwidth\undefined%
    \setlength{\unitlength}{368.00002461bp}%
    \ifx\svgscale\undefined%
      \relax%
    \else%
      \setlength{\unitlength}{\unitlength * \real{\svgscale}}%
    \fi%
  \else%
    \setlength{\unitlength}{\svgwidth}%
  \fi%
  \global\let\svgwidth\undefined%
  \global\let\svgscale\undefined%
  \makeatother%
  \begin{picture}(1,0.43815899)%
    \lineheight{1}%
    \setlength\tabcolsep{0pt}%
    \put(0,0){\includegraphics[width=\unitlength,page=1]{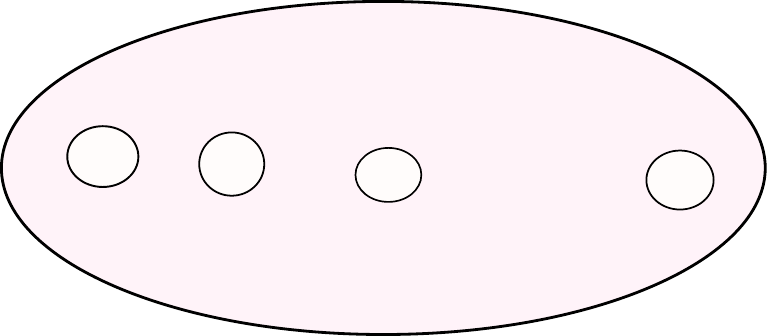}}%
    \put(0.62310429,0.18969373){\color[rgb]{0,0,0}\makebox(0,0)[lt]{\lineheight{1.25}\smash{\begin{tabular}[t]{l}$\cdots$\end{tabular}}}}%
    \put(0.07222315,0.1575609){\color[rgb]{0,0,0}\makebox(0,0)[lt]{\lineheight{1.25}\smash{\begin{tabular}[t]{l}$\gamma_1$\end{tabular}}}}%
    \put(0.24785021,0.13610398){\color[rgb]{0,0,0}\makebox(0,0)[lt]{\lineheight{1.25}\smash{\begin{tabular}[t]{l}$\gamma_2$\end{tabular}}}}%
    \put(0.45389761,0.12138493){\color[rgb]{0,0,0}\makebox(0,0)[lt]{\lineheight{1.25}\smash{\begin{tabular}[t]{l}$\gamma_3$\end{tabular}}}}%
    \put(0.82969312,0.12317247){\color[rgb]{0,0,0}\makebox(0,0)[lt]{\lineheight{1.25}\smash{\begin{tabular}[t]{l}$\gamma_b$\end{tabular}}}}%
    \put(0,0){\includegraphics[width=\unitlength,page=2]{planar.pdf}}%
    \put(0.11773456,0.3313693){\color[rgb]{0,0,0}\makebox(0,0)[lt]{\lineheight{1.25}\smash{\begin{tabular}[t]{l}$\gamma$\end{tabular}}}}%
    \put(0,0){\includegraphics[width=\unitlength,page=3]{planar.pdf}}%
  \end{picture}%
\endgroup%
  
				\caption{}  
				\label{fig-13}  
			\end{figure}  
			
			\textbf{Case 3} ($g=1$ and $b \geq 2$) : Consider the disjoint simple closed curves $a, b$ and $c$ on $\Sigma_{1,g}$, as shown in figure \ref{fig-16}. It is then enough to note that $[a] = [b] + [c] \in H_2(\Sigma_{0,b};\mathbb{Z}_2)$. 
			
				\begin{figure}[!htb]
				\centering
				\def\svgwidth{7cm}
\begingroup%
  \makeatletter%
  \providecommand\color[2][]{%
    \errmessage{(Inkscape) Color is used for the text in Inkscape, but the package 'color.sty' is not loaded}%
    \renewcommand\color[2][]{}%
  }%
  \providecommand\transparent[1]{%
    \errmessage{(Inkscape) Transparency is used (non-zero) for the text in Inkscape, but the package 'transparent.sty' is not loaded}%
    \renewcommand\transparent[1]{}%
  }%
  \providecommand\rotatebox[2]{#2}%
  \newcommand*\fsize{\dimexpr\f@size pt\relax}%
  \newcommand*\lineheight[1]{\fontsize{\fsize}{#1\fsize}\selectfont}%
  \ifx\svgwidth\undefined%
    \setlength{\unitlength}{441.94555793bp}%
    \ifx\svgscale\undefined%
      \relax%
    \else%
      \setlength{\unitlength}{\unitlength * \real{\svgscale}}%
    \fi%
  \else%
    \setlength{\unitlength}{\svgwidth}%
  \fi%
  \global\let\svgwidth\undefined%
  \global\let\svgscale\undefined%
  \makeatother%
  \begin{picture}(1,0.38448898)%
    \lineheight{1}%
    \setlength\tabcolsep{0pt}%
    \put(0,0){\includegraphics[width=\unitlength,page=1]{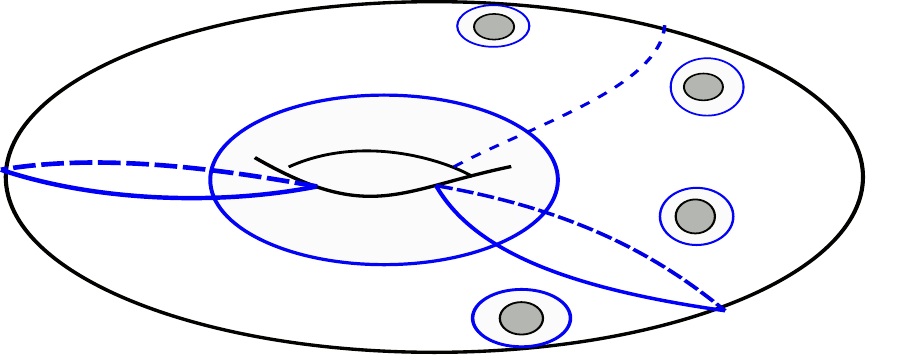}}%
    \put(0.75185722,0.19462249){\color[rgb]{0,0,0}\makebox(0,0)[lt]{\lineheight{1.25}\smash{\begin{tabular}[t]{l}$\vdots$\end{tabular}}}}%
    \put(0.05653523,0.13715881){\color[rgb]{0,0,0}\makebox(0,0)[lt]{\lineheight{1.25}\smash{\begin{tabular}[t]{l}$b$\end{tabular}}}}%
    \put(0.45855961,0.32741298){\color[rgb]{0,0,0}\makebox(0,0)[lt]{\lineheight{1.25}\smash{\begin{tabular}[t]{l}$a$\end{tabular}}}}%
    \put(0.56866373,0.30448538){\color[rgb]{0,0,0}\makebox(0,0)[lt]{\lineheight{1.25}\smash{\begin{tabular}[t]{l}$c$\end{tabular}}}}%
    \put(0,0){\includegraphics[width=\unitlength,page=2]{drawing.pdf}}%
  \end{picture}%
\endgroup%

				\caption{}
				\label{fig-16}
			\end{figure}
		
		\end{enumerate}

	\end{proof}

	\begin{proof}[Proof of Theorem \ref{thm1}] Let us assume that $f : (\Sigma_{g,b}, \partial \Sigma_{g,b}) \rightarrow (W, \partial W)$ is a proper embedding.
		
		\begin{enumerate}
			\item[(1)] We first consider the case with nonempty boundary.  Since $W^4$ is an integral homology $4$-ball with boundary that is an integral homology $3$-sphere, $H_2(W,\partial W; \Z) = 0$. Thus, $f$ is characteristic.  By statement $(c)$ of Lemma \ref{mainlemma}, for $g \geq 1$ and $b \geq2$, there exists an essential simple closed curve $\alpha_1$ in the interior of $W$ such that $q_f([\alpha_1]) = 0$. Therefore, $\tau_{\alpha_1}$ is not $f$-extendible.
			
			\noindent 
			
			The only remaining case is when $g=1$ and $b=0$. Since $W$ is closed and an integral homology $4$-sphere, again $f$ is characteristic. By Rokhlin's theorem, $Arf([\Sigma_{1,0}]) = \frac{\sigma(W) - [\Sigma_{1,0}] \cdot [\Sigma_{1,0}]}{8} = 0 \pmod 2$. If $l$ and $m$ denote the longitudinal and meridinal curves on $\Sigma_{1,0}$ respectively, then $q_f([m])q_f([l]) = Arf([\Sigma_{1,0}]) = 0$. Thus, $q_f$ vanishes for at least one of $l$ and $m$. Say, $q_f([m]) = 0$. Therefore, $q_f([\tau_m(l)]) = q_f([m] + [l]) = q_f([l]) + 1$. Hence, $\tau_m$ is not $f$-extendible.
			
			\vspace{0.1cm}
			
			\item[(2)] For $b \geq 3$, by statement $(c)$ of Lemma \ref{mainlemma}, we can again find a simple closed curve $\beta$ that is parallel to a boundary component of $\Sigma_{g,b}$ and such that $q_f([\beta]) = 0$. By statement $(b)$ of Lemma \ref{mainlemma}, $\tau_\beta$ is not $f$-extendible.
			
		\end{enumerate}
		
	\end{proof}

	\subsection{Proof of Theorem \ref{thm1.5}}
	
	\begin{proof}[Proof of theorem \ref{thm1.5}]
		
		Let $\textrm{OB}(V^4,\phi)$ be a simple open book for $\s^5$ such that $V^4$ is spin. Let $ \Sigma$ be a connected orientable surface of genus $g$ with connected boundary, and let $f : (\Sigma, \partial \Sigma) \rightarrow (V,\partial V)$ be a proper embedding. Without loss of generality, we assume that $(g,b) \neq (1,1)$. We consider two cases.
		
		\begin{enumerate}

			\item[\textbf{Case 1}]  Let $[\Sigma,\partial \Sigma]$ belong to the kernel of the induced homomorphism $f_* : H_2(\Sigma, \partial \Sigma ; \mathbb{Z}) \rightarrow H_2(V, \partial V; \mathbb{Z})$. Since $V$ is spin, $f$ is characteristic. By Lemma \ref{mainlemma}, there exists a simple closed curve $\gamma$ in the interior of $\Sigma$ such that $\tau_\gamma$ is not $f$-extendible. Thus, there is no open book embedding of $\textrm{OB}(\Sigma, \tau_\gamma)$ in $\textrm{OB}(V,\phi)$.

			\item[\textbf{Case 2}] Say $[\Sigma, \partial \Sigma]$ does not lie in the kernel of $f_*$. Let $x \in H_2(V,\partial V ; \mathbb{Z})$ denote the image $f_*([\Sigma,\partial \Sigma])$. Recall that the the variation of the map $\phi$ is given by $\Delta_\phi = (id_V)_* - \phi_*$. Kauffman \cite{Kau} proved the following lemma for simple open books.
			
			\begin{lemma} \label{kauf}
				
				$\textrm{OB}(V,h)$ is a homotopy sphere if and only if $\Delta_h$ is an isomorphism.
				
			\end{lemma}

			\noindent By Lemma \ref{kauf}, $\Delta_\phi = (id_V)_* - \phi_* : H_2(V,\partial V; \mathbb{Z}) \rightarrow H_2(V; \mathbb{Z})$ is an isomorphism. Thus, $\phi_*$ does not fix the homology class $x$. Hence, there can be no open book embedding of $\textrm{OB}(\Sigma,h)$ in $\textrm{OB}(V, \phi)$ for any $h \in \mathcal{MCG}(\Sigma,\partial \Sigma)$.

		\end{enumerate}

	\end{proof}
	
	\noindent

	 The proof of Theorem \ref{thm1.5} also implies the following result about the universality of a $5$-dimensional simple open book with a spin page, in terms of its variation map.

	 \begin{theorem}
	 	
	 	Let $W$ be compact spin $4$-manifold with boundary. A simple open book $\textrm{OB}(W,\psi)$ can not be universal, unless $\Delta_\psi$ has nontrivial kernel.
	 	
	 \end{theorem}

	\section{Examples of extendible/non-extendible mapping classes}\label{sec4}
	
	\subsection{Mapping classes induced by an isotopy of $\s^3 \times [0,1]$} 
	
	A Hopf annulus is an embedding of an annulus in $\mathbb{S}^3$ with boundary a Hopf link. We call a Hopf annulus positive (or negative) if the linking number of its boundary is $1$ (or $-1$).  
	
	\begin{figure}[!htb]
		\centering
		\def\svgwidth{6cm}
\begingroup%
  \makeatletter%
  \providecommand\color[2][]{%
    \errmessage{(Inkscape) Color is used for the text in Inkscape, but the package 'color.sty' is not loaded}%
    \renewcommand\color[2][]{}%
  }%
  \providecommand\transparent[1]{%
    \errmessage{(Inkscape) Transparency is used (non-zero) for the text in Inkscape, but the package 'transparent.sty' is not loaded}%
    \renewcommand\transparent[1]{}%
  }%
  \providecommand\rotatebox[2]{#2}%
  \newcommand*\fsize{\dimexpr\f@size pt\relax}%
  \newcommand*\lineheight[1]{\fontsize{\fsize}{#1\fsize}\selectfont}%
  \ifx\svgwidth\undefined%
    \setlength{\unitlength}{335.27064575bp}%
    \ifx\svgscale\undefined%
      \relax%
    \else%
      \setlength{\unitlength}{\unitlength * \real{\svgscale}}%
    \fi%
  \else%
    \setlength{\unitlength}{\svgwidth}%
  \fi%
  \global\let\svgwidth\undefined%
  \global\let\svgscale\undefined%
  \makeatother%
  \begin{picture}(1,0.3300828)%
    \lineheight{1}%
    \setlength\tabcolsep{0pt}%
    \put(0,0){\includegraphics[width=\unitlength,page=1]{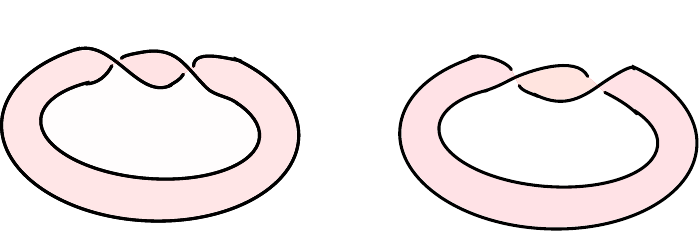}}%
    \put(0.20602709,0.28022802){\color[rgb]{0,0,0}\makebox(0,0)[lt]{\lineheight{1.25}\smash{\begin{tabular}[t]{l}$+$\end{tabular}}}}%
    \put(0.7680343,0.24995874){\color[rgb]{0,0,0}\makebox(0,0)[lt]{\lineheight{1.25}\smash{\begin{tabular}[t]{l}$-$\end{tabular}}}}%
    \put(-0.16097568,0.47801636){\color[rgb]{0,0,0}\makebox(0,0)[lt]{\begin{minipage}{1.16189242\unitlength}\raggedright \end{minipage}}}%
    \put(0.03036771,0.36044391){\color[rgb]{0,0,0}\makebox(0,0)[lt]{\begin{minipage}{0.97285442\unitlength}\raggedright \end{minipage}}}%
    \put(0,0){\includegraphics[width=\unitlength,page=2]{Hopf_band.pdf}}%
    \put(0.1517657,0.02355307){\color[rgb]{0,0,0}\makebox(0,0)[lt]{\lineheight{1.25}\smash{\begin{tabular}[t]{l}$c$\end{tabular}}}}%
    \put(0.72719494,0.01511156){\color[rgb]{0,0,0}\makebox(0,0)[lt]{\lineheight{1.25}\smash{\begin{tabular}[t]{l}$c$\end{tabular}}}}%
  \end{picture}%
\endgroup%

		\caption{Hopf annuli}
		\label{fig-01}
	\end{figure}

	\begin{proposition}[The Hopf annulus trick \cite{hy},\cite{pps}]\label{pr1}
		There exists a flexible proper embedding of an annulus $A$ in $\s^3 \times [0,1]$ such that the boundary of $A$ is embedded in $\s^3 \times \{1\}$. 
	\end{proposition}
	
	\noindent We briefly describe the embedding in proposition \ref{pr1}. The main idea is to observe that $\s^3$ admits an open book decomposition with page a Hopf annulus $\mathcal{A}$ and monodromy a Dehn twist along the core circle $\alpha$ of $\mathcal{A}$. Consider a Seifert type embedding (see subsection \ref{ssec2}) $f_{\mathcal{A}}$ for $\mathcal{A}$ in $\s^3\times [0,1]$ as follows. We smoothly push $\mathcal A$ from $\partial \D^4= \s^3\times \{1\}$ to the level $\s^3\times \{0\}$, keeping its boundary fixed, such that $\s^3\times \{t\} \cap \mathcal A$ is a Hopf link for each $t\in (0,1]$. Let $\Psi_t$ be the isotopy of $\s^3$ such that $\Psi_1$ realizes the Dehn twist monodromy on $\mathcal{A}$ via the flow of the open book $\textrm{OB}(\mathcal{A},\tau_\alpha)$. One can then define an isotopy $\Gamma_s$ ($s\in[0,1]$) of $\s^3 \times  [-1,1]$ such $\Gamma_0 = id_{\s^3 \times [-1,1]}$ and $\Gamma_1|_{\s^3 \times \{0\}} = \Psi_1.$

	$$
	\Gamma_s (x,t) 
	=
	\left\{
	\begin{array}{ll}
		\Psi_{s(1-t)}(x)  & \mbox{if } t \geq 0 \\
		\Psi_{s(t+1)}(x) & \mbox{if } t \leq 0
	\end{array}
	\right . 
	$$
	
	\vspace{0.1cm}
		
	\noindent We then extend $\Gamma_1$ to a diffeomorphism $\Gamma$ of $\s^3 \times [-1,1]$ by the 
	identity map. It is clear that $\Gamma$ is isotopic (relative to boundary) to the identity map on $\s^3 \times [-1,1]$ and $\Gamma$ makes $\tau_{\alpha}$ $f_\mathcal{A}$-extendible. Since $\mathcal{MCG}(\mathcal{A}, \partial \mathcal{A})$ is genrated by $\tau_{\alpha}$, $f_\mathcal{A}$ is also flexible. This argument also shows that the Seifert type embedding for $\mathcal{A}$ is flexible in $\D^4$.
	
	\noindent In general, consider an embedding $f$ of a surface $\Sigma$ in $V^4$ and a simple closed curve $\gamma \subset \Sigma$. Say there is a copy of $\D^2 \times \D^2$ embedded in $V^4$ such that $f(\mathcal{N}_\Sigma(\gamma)) \cap (\D^2 \times \D^2) = f(\mathcal{N}_\Sigma(\gamma)) \cap \partial(\D^2 \times \D^2) = \mathcal{A}$, where $\mathcal{N}_\Sigma(\gamma)$ is a neighborhood of $\gamma$ in $\Sigma$. Then, the Hopf annulus trick implies that $\tau_\gamma$ is $f$-extendible.
	
	\vspace{0.1cm}
	
	\begin{remark}[Convention on Dehn twist and its inverse]
		Here, by $\tau_\alpha$ ($\tau_\alpha^{-1}$) we mean the \emph{left handed} (\emph{right handed}) Dehn twist along the curve $\alpha$. In particular, if $\beta$ is another curve that intersects $\alpha$, then to obtain the curve $\tau_\alpha(\beta)$ one starts from a non-intersecting point on $\beta$ and for every intersection point one takes a left turn to go around $\alpha$ and then follow $\beta$ again. By this convention, we can observe the following change in framing under Dehn twists. Consider the ribbon surfaces $R(\beta, n_\beta)$ and $R(\alpha, n_\alpha)$ and $\alpha$ be an unknot. Then $\tau_\alpha(\beta)$ has framing $n_\beta + n_\alpha + 1$ in the plumbed surface $R(\alpha, n_\alpha) \natural R(\beta, n_\beta).$ Similarly, $\tau_\alpha^{-1}(\beta)$ has surface framing $n_\beta + n_\alpha - 1$. \end{remark}
	
	There is another isotopy of $\s^3 \times [0,1]$ known as the \emph{tube trick}. Let $\mathcal{A}_0$ denote a planar annulus in $\s^3$, i.e., the linking number of the boundary link of $\mathcal{A}_0$ is zero. Let $f_{\mathcal{A}_0}$ denote the Seifert type embedding for $\mathcal{A}_0$ and let $\alpha_0$ be the core curve of $\mathcal{A}_0$.

		\begin{exmp}[Flexible embedding of $\Sigma_{1,1}$ in $\D^4$] \label{flex1}
		The plumbing of two Hopf annuli in $\mathbb{S}^3$ gives a flexible embedding $f$ of $\Sigma_{1,1}$ inside $\mathbb{D}^4$ as shown in figure \ref{fig-1}. This is because $\mathcal{MCG}(\Sigma_{1,1}, \partial \Sigma_{1,1})$ is generated by Dehn twists along the core curves of the Hopf annuli.
		
	 \begin{figure}[!htb]
		\centering
		\def\svgwidth{8cm}
\begingroup%
  \makeatletter%
  \providecommand\color[2][]{%
    \errmessage{(Inkscape) Color is used for the text in Inkscape, but the package 'color.sty' is not loaded}%
    \renewcommand\color[2][]{}%
  }%
  \providecommand\transparent[1]{%
    \errmessage{(Inkscape) Transparency is used (non-zero) for the text in Inkscape, but the package 'transparent.sty' is not loaded}%
    \renewcommand\transparent[1]{}%
  }%
  \providecommand\rotatebox[2]{#2}%
  \newcommand*\fsize{\dimexpr\f@size pt\relax}%
  \newcommand*\lineheight[1]{\fontsize{\fsize}{#1\fsize}\selectfont}%
  \ifx\svgwidth\undefined%
    \setlength{\unitlength}{675.81414056bp}%
    \ifx\svgscale\undefined%
      \relax%
    \else%
      \setlength{\unitlength}{\unitlength * \real{\svgscale}}%
    \fi%
  \else%
    \setlength{\unitlength}{\svgwidth}%
  \fi%
  \global\let\svgwidth\undefined%
  \global\let\svgscale\undefined%
  \makeatother%
  \begin{picture}(1,0.3435983)%
    \lineheight{1}%
    \setlength\tabcolsep{0pt}%
    \put(0,0){\includegraphics[width=\unitlength,page=1]{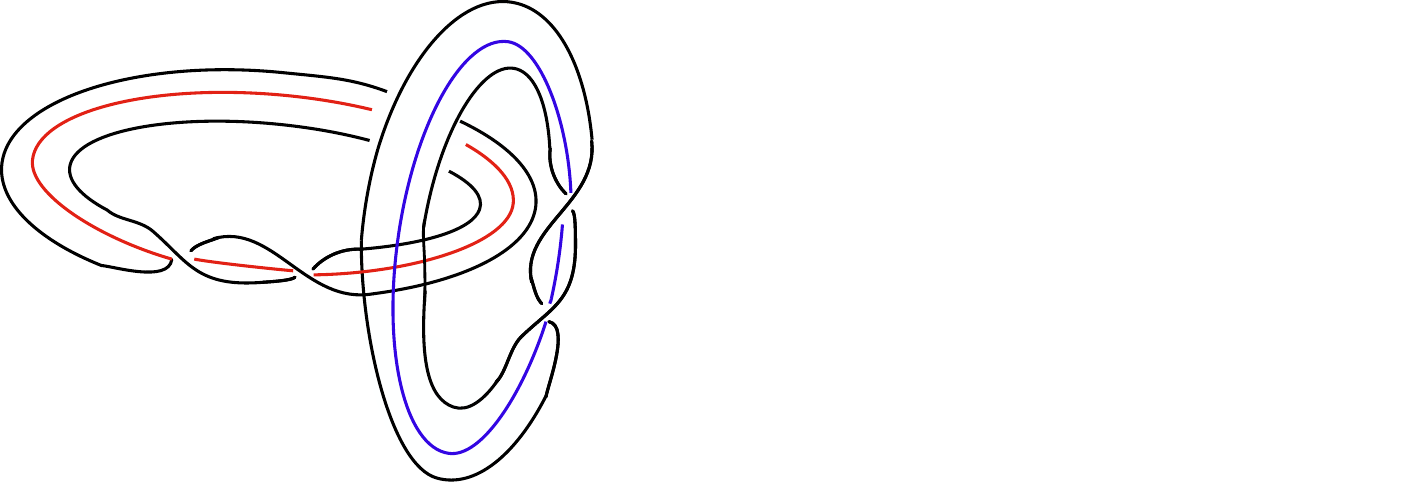}}%
    \put(0.13836673,0.09557565){\color[rgb]{0,0,0}\makebox(0,0)[lt]{\lineheight{1.25}\smash{\begin{tabular}[t]{l}$1$\end{tabular}}}}%
    \put(0.42319075,0.13570346){\color[rgb]{0,0,0}\makebox(0,0)[lt]{\lineheight{1.25}\smash{\begin{tabular}[t]{l}$1$\end{tabular}}}}%
    \put(0,0){\includegraphics[width=\unitlength,page=2]{Trefoil.pdf}}%
    \put(0.71231379,0.08750041){\color[rgb]{0,0,0}\makebox(0,0)[lt]{\lineheight{1.25}\smash{\begin{tabular}[t]{l}$1$\end{tabular}}}}%
    \put(0.96939343,0.1521085){\color[rgb]{0,0,0}\makebox(0,0)[lt]{\lineheight{1.25}\smash{\begin{tabular}[t]{l}$-1$\end{tabular}}}}%
    \put(0,0){\includegraphics[width=\unitlength,page=3]{Trefoil.pdf}}%
  \end{picture}%
\endgroup%

		\vspace{0.3cm}
		\caption{Embeddings of $\Sigma_{1,1}$  in $\mathbb{D}^4$ with boubdaries Trefoil knot and Fig-eight knot}
		\label{fig-1}
	\end{figure}
		
	\end{exmp}
	
	More generally, plumbing of two $\pm 1$-framed ribbons of slice knots in $\mathbb{S}^3$  gives a flexible embedding of $\Sigma_{1,1}$ inside $\mathbb{D}^4$. Note that by statement $(a)$ of Theorem \ref{thm1}, the boundary knots for these embeddings can not be slice in any integral homology $4$-ball with boundary $\s^3$.

	\begin{proposition}[The tube trick \cite{hirose2}] \label{tubetrick}
	
	There is a diffeomorphism $H_1$ of $\s^3 \times [0,1]$, relative to boundary, such that $H_1$ is isotopic to the identity map and $H_1\circ f_{\mathcal{A}_0} = f_{\mathcal{A}_0} \circ \tau^2_{\alpha_0}$.    
		
	\end{proposition}
	
	In general, consider an embedding $f'$ of a surface $\Sigma$ in $V^4$ and a simple closed curve $\beta \subset \Sigma$ such that there exists  a copy of $\D^2 \times \D^2$ embedded in $V^4$ such that $f'(\mathcal{N}_\Sigma(\gamma)) \cap (\D^2 \times \D^2) = f'(\mathcal{N}_\Sigma(\beta)) \cap \partial(\D^2 \times \D^2) = \mathcal{A}_0$. Here $\mathcal{N}_\Sigma(\beta)$ is a neighborhood of $\beta$ in $\Sigma$. The tube trick then implies that $\tau^2_\beta$ is $f'$-extendible.

	\vspace{0.1cm}
	
	Recall the definition of a ribbon embedding $\mathcal{R}(K,n)$ for a knot $K$ with framing $n$ (subsection\ref{ssec2}). The following theorem gives some more examples of flexible and non-flexible embeddings. Let $int(X)$ denote the interior of $X$.
	
		\begin{theorem}\label{thm2} 
		
		Let $W^4$ be a $4$-manifold with boundary $\s^3$.
		
		\begin{enumerate}
			\item[(a)]  For a knot $K \subset \s^3$ and for $m \in \Z$, $\mathcal{R}(K,2m)$ is not flexible in $W^4$.
			\item[(b)]  If $K$ is slice in $W^4$, then $\mathcal{R}(K,\pm1)$ is flexible in $W^4$.
			\item[(c)]  There exist non-characteristic proper flexible embeddings of $\Sigma_{g,1}$ in $\CP^2 \setminus int(\D^4)$ and disk bundles over $\s^2$ with euler number $\pm3$. 
		\end{enumerate}
		
	\end{theorem}  
	
	\noindent Statement $(c)$ in Theorem \ref{thm2} was already known due to the works of Hirose--Yasuhara \cite{hy} and Etnyre--Lekili \cite{EL}.
	
	\vspace{0.25cm}
	
	\noindent Since every knot in $\s^3$ is slice in $\s^2 \times \s^2 \setminus int(\D^4)$, Theorem \ref{thm2} implies the following.
	
	\begin{corollary} \label{cor2}
		For a knot $K \subset \s^3$, $\mathcal{R}(K,\pm1)$ is flexible in $\s^2 \times \s^2 \setminus int(\D^4)$.
	\end{corollary}

	We need the following lemma from \cite{hy}.

	\begin{lemma} \label{lemmai3}
		
		Let $V^4$ be a smooth $4$-manifold with boundary $\s^3$. Let $K$ be a knot in $\s^3$ such that $K$ bounds the core disk of an $n$-framed $2$-handle embedded in $V^4$.  Then $\mathcal{R}(K,n\pm1)$ is flexible in $V^4$.
	\end{lemma}

	\begin{proof}
		Consider a collar neighborhood of the boundary $\s^3$ in $V$, diffeomorphic to $\s^3 \times [0,1]$ such that $\s^3 \times \{1\} = \partial V$. Let $c$ be the core circle of the $(n \pm 1)$--framed ribbon ($\subset \mathcal{R}(K,n\pm1)$) at the level $\s^3 \times \{0\}$. Since $K$ bounds the core disk of an $n$-framed $2$-handle, there is a $2$-disk $D^2$ embedded in $V^4$ such that $c = \partial D^2$ and $D^2 \cap \mathcal{R}(K, n\pm1)=c$. The normal bundle, $\nu_{D^2}$, is isomorphic to $\mathbb{D}^2 \times \mathbb{D}^2$. Since the surface framing of $c$ is $n \pm1$, $\nu_{\mathbb{D}^2} \cap \mathcal{R}(K, n\pm1)$ is a Hopf annulus in the boundary of $\nu_ {D^2}$ with core curve $c$. By the Hopf annulus trick, $\tau_c$ is extendible and hence, $\mathcal{R}(K,n\pm1)$ is flexible.
	\end{proof}

	\begin{proof}[Proof of Theorem \ref{thm2}] Given two framed knots $(K_1, n_1)$ and $(K_2, n_2)$, let $R(K_1,n_1)$ and $R(K_2,n_2)$ be their corresponding ribbon surfaces respectively. Let $P((K_1,n_1), (K_2,n_2)) = R(K_1,n_1) \natural R(K_2,n_2)$ denote the surface obtained by plumbing $R(K_1,n_1)$ and $R(K_2,n_2)$. Let $\mathcal{R}((K_1,n_1), (K_2,n_2))$ denote the Seifert type embedding in $\s^3 \times [-1,1]$ corresponding to $P((K_1,n_1), (K_2,n_2))$. Observe that if $\mathcal{R}(K_1,n_1)$ and $\mathcal{R}(K_2,n_2)$ are flexible embeddings then so is the embedding $\mathcal{R}((K_1,n_1), (K_2,n_2))$. It can be seen in the following way. We can isotope $\mathcal{R}((K_1,n_1), (K_2,n_2))$ so that $P((K_1,n_1), (K_2,n_2))$ is contained at the level $\subset \s^3 \times \{0\}$ of the collar $\s^3 \times [-1,1]$. We first isotope a closed collar neighborhood $\mathcal{N}(K_1)\subset P((K_1,n_1), (K_2,n_2))$ of $K_1$ to the level $\s^3 \times \{ \frac{1}{2} \}$, relative to $\partial \mathcal{N}(K_1)$. Since the Dehn twist $\tau_{K_1}$ is extendible for $\mathcal{R}((K_1,n_1), (K_2,n_2))$, one can apply a diffeomorphism $\Phi_1$ of $W \setminus (\s^3 \times (0,1])$ that is identity near boundary and that induces $\tau_{K_1}$. Here, we think of $K_1$ as the core curve of the collar $\mathcal{N}(K_1)$. We then isotope $\mathcal{N}(K_1)$ back to its original position and repeat the same process for a closed collar neighborhood of $K_2$ in $P((K_1,n_1), (K_2,n_2))$ to get a diffeomorphism $\Phi_2$ that induces the Dehn twist $\tau_{K_2}$. Combining $\Phi_1$, $\Phi_2$ along with the ambient isotopies of $V$, we can induce any mapping class generated by $\tau_{K_1}$ and $\tau_{K_2}$ for the embedding $\mathcal{R}((K_1,n_1), (K_2,n_2))$. Since $\mathcal{MCG}(\Sigma_{1,1}, \partial \Sigma_{1,1})$ is generated by $\tau_{K_1}$ and $\tau_{K_2}$, $P((K_1,n_1), (K_2,n_2))$ is flexible. 
		
	\noindent We now prove the statements in Theorem \ref{thm2}.

		\begin{enumerate}
			\item[(a)] Say $\mathcal{R}(K,2m)$ is flexible in $W$. By plumbing two copies of $\mathcal{R}(K,2m)$, we get a flexible proper embedding $h$ of $\Sigma_{1,1}$ in $W$ such that $h$ is characteristic. From the examples in subsection \ref{ssec2}, one can see that $q_f([K]) = 0$. By Lemma \ref{mainlemma}, we get a contradiction. 
			
			\vspace{0.1cm}
			
			\item[(b)] Since $K$ is slice in $W$, $K$ bounds the core disk of a $0$-framed $2$-handle in $W^4$. The flexibility of $\mathcal{R}(K,\pm1)$ then follows from Lemma \ref{lemmai3}.
			
			\vspace{0.1cm}
			
			\item[(c)] Consider the set of curves $\mathcal{H} = \{a_1,c_1,a_2,c_2,...a_{g-1},c_{g-1},a_g,b_1,b_2\}$ on $\Sigma_{g,1}$, as shown in figure \ref{humphreygen}. Consider the surface $S$ in $\s^3$ obtained by consecutively plumbing $2g$ copies of a negative Hopf annulus, $R(U, -1)$ (see figure \ref{fig-12}). Let $f_S$ be a Hopf type embedding in $\s^3 \times [-1,1]$ (boundary embeds in $\s^3 \times \{1\}$) corresponding to $S$. Thus, $f_S$ gives an embedding of $\Sigma_{g,1}$ such that all the curves in $\mathcal{H}$, except $b_2$, has surface framing $-1$. The curve $b_2$ has surface framing $-2$. We now attach a $2$-handle to $\s^3 \times \{-1\}$ along an unknot with framing $-1$ and cap off the new boundary component (a copy of $\s^3$). This gives a proper embedding $\tilde{f}_S$ of  $\Sigma_{g,1}$ in $\CP^2 \setminus int(\D^4)$, with same surface framings for the curves in $\mathcal{H}$. By the Hopf annulus trick, $\tau_{\gamma}$ is $\tilde{f}_S$-extendible for all $\gamma \in \mathcal{H} \setminus \{b_2\}$. It is clear that the curve $b_2$ bounds the core disk of a $2$-handle with framing $-1$ in $\CP^2 \setminus int(\D^4)$. Therefore, by Lemma \ref{lemmai3}, $\tau_{b_2}$ is also $\tilde{f}_S$-extendible. Since Dehn twists along the curves in $\mathcal{H}$ generate $\mathcal{MCG}(\Sigma_{g,1}, \partial \Sigma_{g,1})$, $\tilde{f}_S$ is flexible.

			\begin{figure}[htbp] 
				
				\centering
				\def\svgwidth{14cm}
				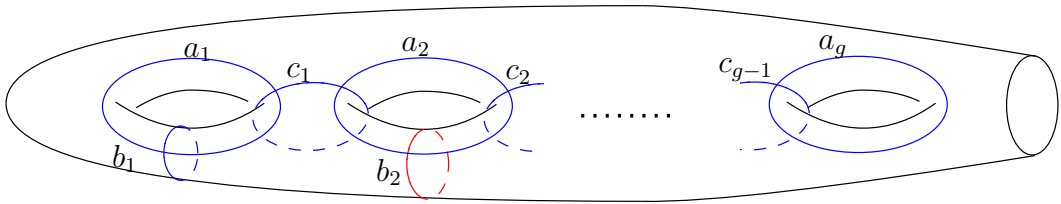
				\caption{ The Humphreys generating set $\mathcal{H} = \{a_1,c_1,a_2,c_2,...a_{g-1},c_{g-1},a_g,b_1,b_2\}$}
			\label{humphreygen}

			\end{figure}
		\end{enumerate}
		
	\end{proof}

	\noindent Our next theorem gives more examples of extendible Dehn twists along a curve on an embedded surface. 
	
	\begin{theorem} \label{thm3}
		Let $f: F \rightarrow \D^4$ be a Seifert type embedding, and let $\alpha, \beta \subset F$ be unknots (in $\s^3$) such that $|\alpha \cap \beta| = 1$.
		\begin{enumerate}
			\item[(a)] Say the surface framing of $\alpha$ is odd and the surface framing of $\beta$ is zero. If $\tau_\beta (\alpha)$ is unknot, then $\tau_{\alpha}$ is $f$-extendible.
			\item[(b)] Say $\alpha$ and $\beta$ have surface framings $2n+1$ and $-k(2n+1)$, respectively, for some $n,k \in \mathbb{Z}$ and $k$ odd. If $ \tau_\alpha$ is $f$-extendible and $\tau_\alpha(\beta)$ is unknot, then $\tau_\beta$ is $f$-extendible,
			\item[(c)] Say the surface framings of $\alpha$ and $\beta$ are $ 2n+1 $ and $-(2n+3)$, respectively. If $\tau_\alpha(\beta)$ (and $\tau_\beta(\alpha)$) is unknot, then $\tau_{\alpha}$ is $f$-extendible if and only if $\tau_{\beta}$ is $f$-extendible.
		\end{enumerate}
	\end{theorem}
	
	\begin{proof}
		\begin{enumerate}
			\item[(a)] Suppose that  the surface framing of $\alpha$ is $2m+1$, for some $m \in \mathbb{Z}.$ Then surface framing of the unknotted curve $\tau^{2k}_\beta(\alpha)$ is $2m + 1 + 2k$, and the surface framing of the unknotted curve $\tau^{-2k}_\beta(\alpha)$ is $2m+1-2k$. We assume that $2m+1 > 0$. Since, $|\alpha \cap \beta| = 1$, a tubular neighborhood of $f(\alpha) \cup f(\beta)$ is istopic to the $R(\alpha, 2m+1) \natural R(\beta, 0)$ (say in $\s^3 \times \{\frac{1}{2}\}$). By the tube trick, there is an isotopy of $\D^4$ (relative to boundary) $H_t$ ($t \in [0,1]$) such that $H_1$ induces the mapping class $\tau_{\beta}^{-2m}$ on the embedded surface. We note that the surface framing of the curve $\tau_{\beta}^{-2m}(\alpha)$ is $1$. Thus, a tubular neighborhood of $\tau_\beta^{-2m}(\alpha)$ on the surface is a positive Hopf annulus in $\mathbb{S}^3 \times \{\frac{1}{2}\}$. Thus, $\tau_\alpha$ is $H_1 \circ f$-extendible. Then, by Lemma \ref{4.4},  $\tau_\beta $ is also $f$-extendible. For $2m+1 < 0 $, we apply the tube trick to induce $ \tau_{\beta}^{2m}$ on the embedded surface and use a similar argument to isotope a tubular neighborhood of $\beta$ into a negative Hopf annulus.
			
			\vspace{0.1cm}
			
			\item[(b)] We assume that $2n+1 >0$ and $-k(2n+1) < 0$. Thus, surface framing of the unknotted curve $\tau_\alpha^k(\beta)$ is $1$. Since $\tau_\alpha$ is $f$-extendible, there exists a diffeomorphism $G$ of $\D^4$ (relative to boundary) such that $G \circ f = f \circ \tau_\alpha^k$. Since the tubular neighborhood of $\tau_{\alpha}^k(\beta)$ now is a Hopf annulus in the embedded surfaces, by the Hopf annuus trick, there is a relative isotopy $F_t$ of $\D^4$ such that $F_1 \circ G \circ f = G \circ f \circ \tau_\beta$. Thus, $G^{-1} \circ F_1 \circ G \circ f = f \circ \tau_\beta$. Hence, by Lemma \ref{4.4}, $\tau_\beta$ is $f$-extendible.
			
			\vspace{0.1cm}
			
			\item[(c)] Assume that $\tau_\alpha$ is $f$-extendible and $n > 0$. Then, the surface framing of the unknotted curves $\tau_\alpha(\beta)$ and $\tau_\beta(\alpha)$ is $-1$. The result then follows from similar arguments as in statement $(b)$ above.
		\end{enumerate}
	\end{proof}

	\begin{figure}[htbp] 
		
		\centering
		\def\svgwidth{14cm}
\begingroup%
  \makeatletter%
  \providecommand\color[2][]{%
    \errmessage{(Inkscape) Color is used for the text in Inkscape, but the package 'color.sty' is not loaded}%
    \renewcommand\color[2][]{}%
  }%
  \providecommand\transparent[1]{%
    \errmessage{(Inkscape) Transparency is used (non-zero) for the text in Inkscape, but the package 'transparent.sty' is not loaded}%
    \renewcommand\transparent[1]{}%
  }%
  \providecommand\rotatebox[2]{#2}%
  \newcommand*\fsize{\dimexpr\f@size pt\relax}%
  \newcommand*\lineheight[1]{\fontsize{\fsize}{#1\fsize}\selectfont}%
  \ifx\svgwidth\undefined%
    \setlength{\unitlength}{1068.13593111bp}%
    \ifx\svgscale\undefined%
      \relax%
    \else%
      \setlength{\unitlength}{\unitlength * \real{\svgscale}}%
    \fi%
  \else%
    \setlength{\unitlength}{\svgwidth}%
  \fi%
  \global\let\svgwidth\undefined%
  \global\let\svgscale\undefined%
  \makeatother%
  \begin{picture}(1,0.23977924)%
    \lineheight{1}%
    \setlength\tabcolsep{0pt}%
    \put(0,0){\includegraphics[width=\unitlength,page=1]{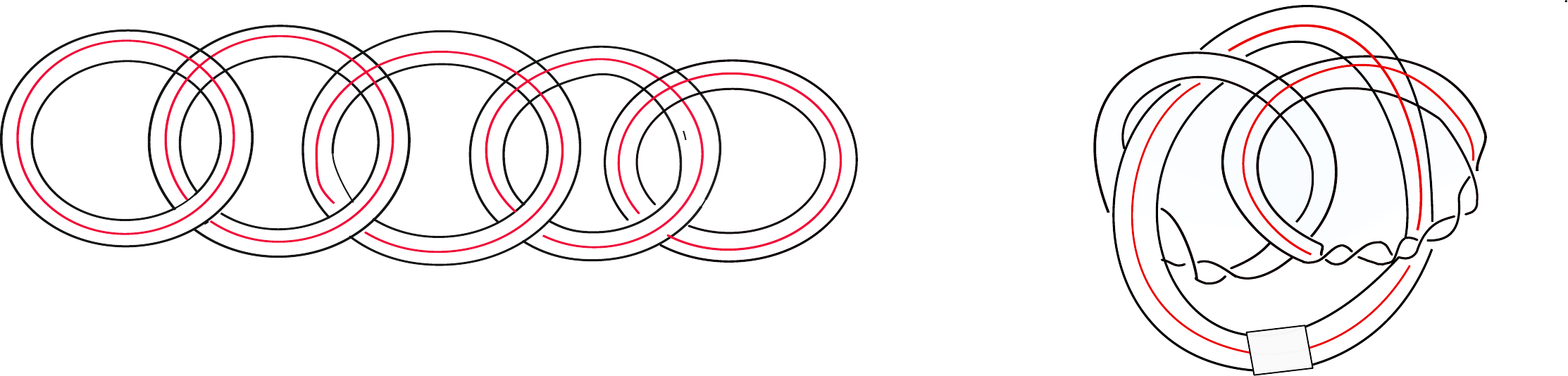}}%
    \put(0.81280024,0.01099976){\color[rgb]{0,0,0}\rotatebox{6.70261666}{\makebox(0,0)[lt]{\lineheight{1.25}\smash{\begin{tabular}[t]{l}$5$\end{tabular}}}}}%
    \put(0,0){\includegraphics[width=\unitlength,page=2]{links.pdf}}%
    \put(0.04393907,0.08431403){\color[rgb]{0,0,0}\rotatebox{1.7161115}{\makebox(0,0)[lt]{\lineheight{1.25}\smash{\begin{tabular}[t]{l}$-1$\end{tabular}}}}}%
    \put(0.16810947,0.08025424){\color[rgb]{0,0,0}\rotatebox{1.7161115}{\makebox(0,0)[lt]{\lineheight{1.25}\smash{\begin{tabular}[t]{l}$3$\end{tabular}}}}}%
    \put(0.26826427,0.07705399){\color[rgb]{0,0,0}\rotatebox{1.7161115}{\makebox(0,0)[lt]{\lineheight{1.25}\smash{\begin{tabular}[t]{l}$-5$\end{tabular}}}}}%
    \put(0.37140052,0.07864352){\color[rgb]{0,0,0}\rotatebox{1.7161115}{\makebox(0,0)[lt]{\lineheight{1.25}\smash{\begin{tabular}[t]{l}$7$\end{tabular}}}}}%
    \put(0.46269742,0.07684342){\color[rgb]{0,0,0}\rotatebox{1.7161115}{\makebox(0,0)[lt]{\lineheight{1.25}\smash{\begin{tabular}[t]{l}$-9$\end{tabular}}}}}%
  \end{picture}%
\endgroup%

		\caption{}
		\label{extexmp}
		
	\end{figure}

	\begin{exmp}
		
		Consider the genus $2$ Seifert surface bounding a $2$-component link, as shown on the left of figure \ref{extexmp} and the corresponding Seifert type embedding in $\s^3 \times [0,1]$. The integers inside the boxes denote the number of full twists in the plumbed annuli. Let $c_i$ ($1\leq i \leq 5$) denote the core curves of these annuli, from left to right. By Theorem \ref{thm3}, $\tau_{c_2}$ is extendible. We can then apply Theorem \ref{thm3} for the curves $c_2$ and $c_3$ to see that $\tau_{c_3}$ is also extendible. In this way, we can see that every $\tau_{c_i}$ is extendible.
	
	\end{exmp}
	
	\begin{exmp}
		
		Consider the genus $1$ Seifert surface bounding a $2$-component link, as shown on the right of figure \ref{extexmp}. Let $a,b,c$ denote the core curves of the three plumbed annuli with twist numbers $1, -3, 5$, respectively. Note that both $\tau_a(b)$ and $\tau_b(c)$ are unknotted. Therefore, by Theorem \ref{thm3}, $\tau_a$, $\tau_b$ and $\tau_c$ are extendible. 
		
	\end{exmp}

	\subsection{Embeddings of $\Sigma_{g,1}$ and Hammenst\"{a}dt generators }
	
	Hirose \cite{hirose1} showed that for the trivial embedding of $\Sigma_g$ in $\mathbb{S}^4$, $\phi \in \mathcal{MCG}(\Sigma_g)$ is extendible if and only if it preserves the Rokhlin quadratic form for $g \geq 1$. He also proved that the \emph{spin mapping class group} of a closed genus-$g$ surface is generated by finitely many Dehn twists. Later, Hammenst\"{a}dt \cite{ursula} provided two generating sets, $\SC_g = \{c_1, c_2, a_1, a_2, \ldots, a_g, b_1, b_2, \ldots, b_{g-1}\}$ and $\SV_g = \{c_1, c_2, c_3, c_4, a_1, a_2, \ldots, a_g, d_1, d_2, \ldots, d_{g-3}\}$, consisting of $2g + 1$ non-separating simple closed curves on a $\Sigma_g$ (see figure \ref{fig4} and figure \ref{fig5}). The curves in $SC_g$ generate the spin mapping class group for odd spin structures on $\Sigma_g$ for $g \geq 3$. The curves in $\SV_g$ generate the spin mapping class group for even spin structures on $\Sigma_g$ for $g\geq4$. We call these curves \emph{Hammenst\"{a}dt generators} or \emph{Hammenst\"{a}dt curves}. It seems natural to ask whether there exists an embedding of $\Sigma_g$ in $\s^4$ that realizes Dehn twists along the Hammenst\"{a}dt curves as extendible elements. We note that the method used in \cite{ursula} to find the Hammenst\"{a}dt curves does not involve codimension $2$ embedding of surfaces and therefore, does not give information about the realizability problem. While it is not clear to us whether such an embedding of $\Sigma_g$ exists in $\s^4$, we can construct a Seifert type embedding of $\Sigma_{g,1}$ in $\mathbb{D}^4$ such that the Hammenst\"{a}dt curves are extendible. It is enough to describe the corresponding Seifert surfaces in $\s^3$. 
	
		\begin{figure}[htbp]
		\centering
		\def\svgwidth{12cm}
		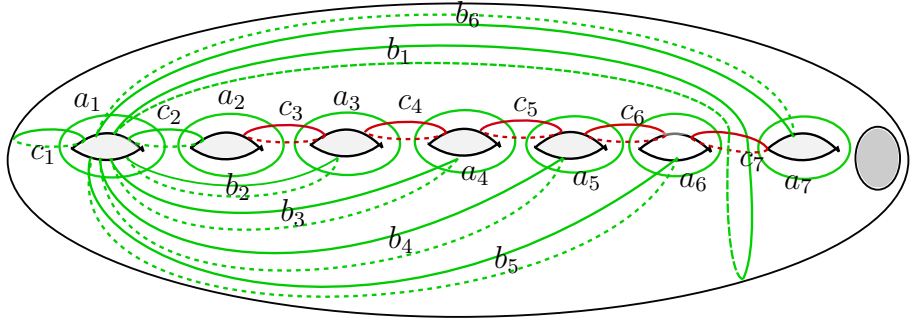
		\caption{Hammenst\(\ddot{a}\)dt curves in $\Sigma_7$ for an odd spin structure.}
		\label{fig4}
	\end{figure} 
	
	\begin{figure}[htbp]
		\centering
		\def\svgwidth{12cm}
		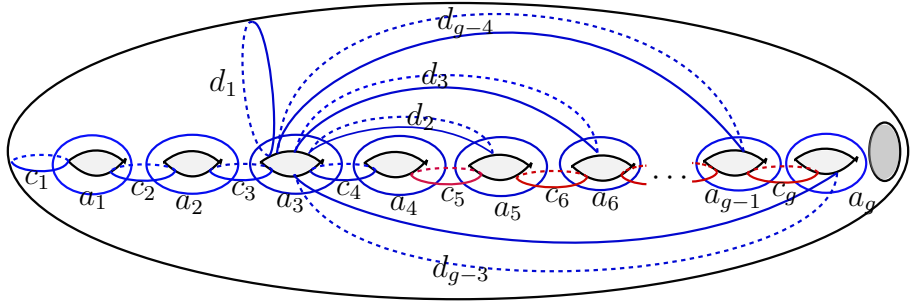
		\caption{Hammenst\(\ddot{a}\)dt curves in $\Sigma_g$ for an even spin structure.}
		\label{fig5}
	\end{figure}

	\smallskip
	\noindent
	\begin{theorem} \label{thmham}
		There exist proper embeddings $f_o$ and $f_e$ of $\Sigma_{g,1}$ in $\mathbb{D}^4$ such that:
		\begin{enumerate}
			\item[(a)] Dehn twists along Hamenst\"{a}dt curves for an odd spin structure on $\Sigma_{g,1}$ are $f_o$-extendible for $g \geq 3$.
			\item[(b)] Dehn twists along Hamenstädt curves for an even spin structure on $\Sigma_{g,1}$ are $f_e$-extendible for $g \geq 4$.
		\end{enumerate}
	\end{theorem}
	
		\begin{figure}[htbp]
		\centering
		\def\svgwidth{9cm}
\begingroup%
  \makeatletter%
  \providecommand\color[2][]{%
    \errmessage{(Inkscape) Color is used for the text in Inkscape, but the package 'color.sty' is not loaded}%
    \renewcommand\color[2][]{}%
  }%
  \providecommand\transparent[1]{%
    \errmessage{(Inkscape) Transparency is used (non-zero) for the text in Inkscape, but the package 'transparent.sty' is not loaded}%
    \renewcommand\transparent[1]{}%
  }%
  \providecommand\rotatebox[2]{#2}%
  \newcommand*\fsize{\dimexpr\f@size pt\relax}%
  \newcommand*\lineheight[1]{\fontsize{\fsize}{#1\fsize}\selectfont}%
  \ifx\svgwidth\undefined%
    \setlength{\unitlength}{932.00478004bp}%
    \ifx\svgscale\undefined%
      \relax%
    \else%
      \setlength{\unitlength}{\unitlength * \real{\svgscale}}%
    \fi%
  \else%
    \setlength{\unitlength}{\svgwidth}%
  \fi%
  \global\let\svgwidth\undefined%
  \global\let\svgscale\undefined%
  \makeatother%
  \begin{picture}(1,0.21740718)%
    \lineheight{1}%
    \setlength\tabcolsep{0pt}%
    \put(0,0){\includegraphics[width=\unitlength,page=1]{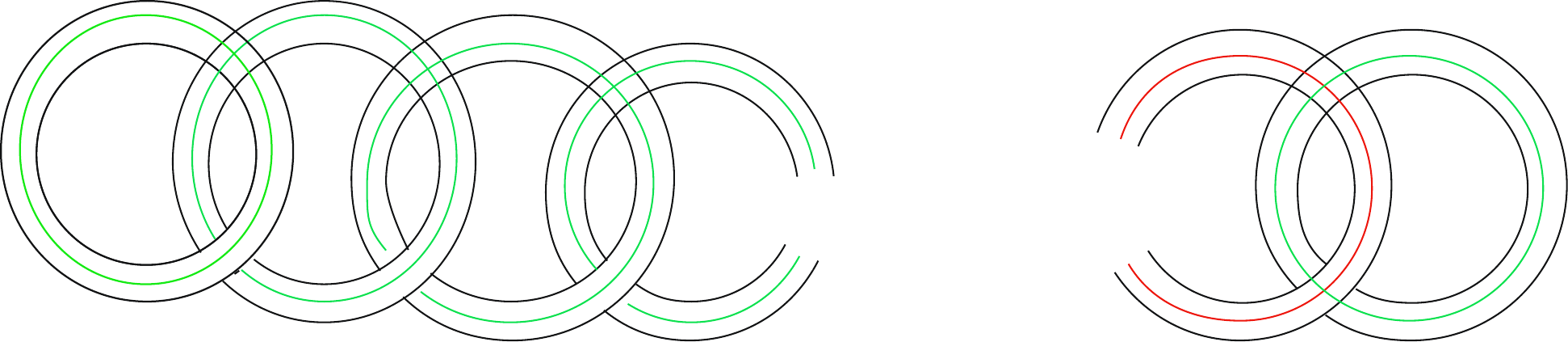}}%
    \put(0.07612454,0.06776067){\color[rgb]{0,0,0}\makebox(0,0)[lt]{\lineheight{1.25}\smash{\begin{tabular}[t]{l}$c_1$\end{tabular}}}}%
    \put(0.19693074,0.05503357){\color[rgb]{0,0,0}\makebox(0,0)[lt]{\lineheight{1.25}\smash{\begin{tabular}[t]{l}$a_1$\end{tabular}}}}%
    \put(0.31487697,0.04112853){\color[rgb]{0,0,0}\makebox(0,0)[lt]{\lineheight{1.25}\smash{\begin{tabular}[t]{l}$c_2$\end{tabular}}}}%
    \put(0.43913473,0.03698202){\color[rgb]{0,0,0}\makebox(0,0)[lt]{\lineheight{1.25}\smash{\begin{tabular}[t]{l}$a_2$\end{tabular}}}}%
    \put(0.76767876,0.041944){\color[rgb]{0,0,0}\makebox(0,0)[lt]{\lineheight{1.25}\smash{\begin{tabular}[t]{l}$c_g$\end{tabular}}}}%
    \put(0.89612715,0.04397207){\color[rgb]{0,0,0}\makebox(0,0)[lt]{\lineheight{1.25}\smash{\begin{tabular}[t]{l}$a_g$\end{tabular}}}}%
    \put(0.55997199,0.07309243){\color[rgb]{0,0,0}\makebox(0,0)[lt]{\lineheight{1.25}\smash{\begin{tabular}[t]{l}$\cdots$\end{tabular}}}}%
  \end{picture}%
\endgroup%

		\caption{ }
		\label{fig12}
	\end{figure}
	
	\begin{figure}[htbp]
		\centering
		\def\svgwidth{10cm}
		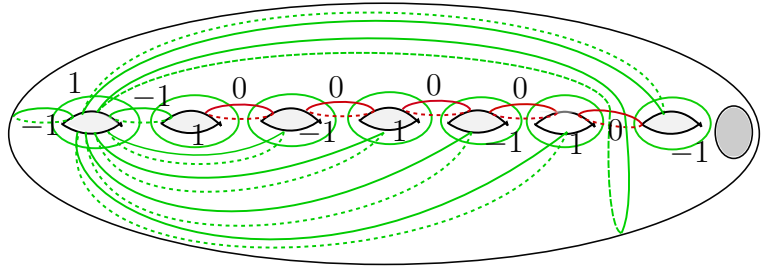
		\caption{The green curves denote the elements of $\SC_g$. }
		\label{oddham}
	\end{figure}

\begin{proof}
	Let $a_i $ and $c_i (1 \leq i \leq g )$ be the curves on $\Sigma_{g,1}$ as illustrated in the figure \ref{fig4} and figure \ref{fig5}. Let $ A_i$ and $C_i$ be the annular neighborhoods of $a_1$ and $c_i$, respectively. Note that $\Sigma_{g,1}$ is diffeomorphic to the plumbing $(A_1 \# C_1 \# A_2 \# C_2 \# \dots \# A_g \# C_g)$, as shown in figure \ref{fig12}.

\begin{enumerate}
	\item[(a)] Consider an embedding $f$ of these plumbed annuli as described in figure \ref{oddham}. Here, the twisted part in each $A_i$ is positioned such that $b_{g-1}$ does not pass through those twisted parts in the $A_i$s (see figure \ref{figb2}). We label the surface framings of the chain curves by $1$, $-1$ and $0$ as shown in figure \ref{oddham}.
	
	\begin{figure}[!htb]
		\centering
		\def\svgwidth{10cm}
		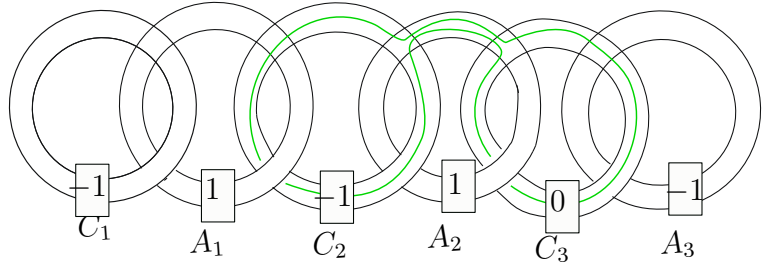
		\caption{Isotopic $b_2$ into a tubular neiborhood of the plumbing.}
		\label{figb2}
	\end{figure}
\begin{figure}[!htb]
	\centering
	\def\svgwidth{10cm}
\begingroup%
  \makeatletter%
  \providecommand\color[2][]{%
    \errmessage{(Inkscape) Color is used for the text in Inkscape, but the package 'color.sty' is not loaded}%
    \renewcommand\color[2][]{}%
  }%
  \providecommand\transparent[1]{%
    \errmessage{(Inkscape) Transparency is used (non-zero) for the text in Inkscape, but the package 'transparent.sty' is not loaded}%
    \renewcommand\transparent[1]{}%
  }%
  \providecommand\rotatebox[2]{#2}%
  \newcommand*\fsize{\dimexpr\f@size pt\relax}%
  \newcommand*\lineheight[1]{\fontsize{\fsize}{#1\fsize}\selectfont}%
  \ifx\svgwidth\undefined%
    \setlength{\unitlength}{373.63815712bp}%
    \ifx\svgscale\undefined%
      \relax%
    \else%
      \setlength{\unitlength}{\unitlength * \real{\svgscale}}%
    \fi%
  \else%
    \setlength{\unitlength}{\svgwidth}%
  \fi%
  \global\let\svgwidth\undefined%
  \global\let\svgscale\undefined%
  \makeatother%
  \begin{picture}(1,0.35831784)%
    \lineheight{1}%
    \setlength\tabcolsep{0pt}%
    \put(0,0){\includegraphics[width=\unitlength,page=1]{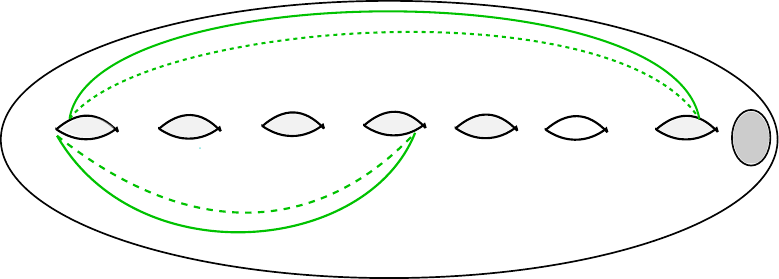}}%
    \put(0.39344617,0.31505446){\color[rgb]{0,0,0}\makebox(0,0)[lt]{\lineheight{1.25}\smash{\begin{tabular}[t]{l}$b_{g-1}$\end{tabular}}}}%
    \put(0.30176293,0.04050602){\color[rgb]{0,0,0}\makebox(0,0)[lt]{\lineheight{1.25}\smash{\begin{tabular}[t]{l}$b_2$\end{tabular}}}}%
    \put(0.28866654,0.1160143){\color[rgb]{0,0,0}\makebox(0,0)[lt]{\lineheight{1.25}\smash{\begin{tabular}[t]{l}$b'_2$\end{tabular}}}}%
    \put(0.46011685,0.23603733){\color[rgb]{0,0,0}\makebox(0,0)[lt]{\lineheight{1.25}\smash{\begin{tabular}[t]{l}$b'_{g-1}$\end{tabular}}}}%
    \put(0,0){\includegraphics[width=\unitlength,page=2]{Isotopy_of_bi.pdf}}%
  \end{picture}%
\endgroup%

	\caption{}
	\label{fig7}
\end{figure}
	
	It is clear that the Dehn twists along the curves $a_i$s and $c_i$s are $f$-extendible. Consider the curve $b_{g-1}'$ (see figure \ref{fig7}) on $\Sigma_{g,1}$ such that $b_{g-1}'$ passes through $C_2$ once, $A_2$ twice, $C_3$ once and so on. It is not hard to see that $b_{g-1}'$ has surface framing $-1$. Thus, $\tau_{b_{g-1}'}$ is $f$-extendible. Since $b_{g-1}'$ is , by lemma \ref{Leic}, $\tau_{b_{g-1}}$ is $f$-extendible. 
	
	\noindent Similarly, we can consider the curve $b_j'$ ($3 \leq j \leq g-1$) isotopic to $b_j$ on $\Sigma_{g,1}$. The curve $b_j'$ passes through $C_2$ once, $A_2$ twice, $C_3$ once, $A_3$ twice, and so on, finally passing through $A_{j-1}$ twice and $C_j$ once. Again, by adding up the framings, we see that the surface framing of $b_j'$ is $1$ for $j$ odd, and $-1$ for $j$ even. Thus, $\tau_{b_j}$ is $f$-extendible for $3 \leq j \leq g-1$. The curve $b_1$ is isotopic to the curve $b_1'$, shown in the figure \ref{fig8} and a similar analysis shows that $b_1'$ has surface framing $1$ for $g$ odd, and $-1$ for for $g$ even. Thus, $\tau_{b_1}$ is $f$-extendible. 
		
	\begin{figure}[!htb]
		\centering
		\vspace{-5cm}
		\def\svgwidth{11cm}
\begingroup%
  \makeatletter%
  \providecommand\color[2][]{%
    \errmessage{(Inkscape) Color is used for the text in Inkscape, but the package 'color.sty' is not loaded}%
    \renewcommand\color[2][]{}%
  }%
  \providecommand\transparent[1]{%
    \errmessage{(Inkscape) Transparency is used (non-zero) for the text in Inkscape, but the package 'transparent.sty' is not loaded}%
    \renewcommand\transparent[1]{}%
  }%
  \providecommand\rotatebox[2]{#2}%
  \newcommand*\fsize{\dimexpr\f@size pt\relax}%
  \newcommand*\lineheight[1]{\fontsize{\fsize}{#1\fsize}\selectfont}%
  \ifx\svgwidth\undefined%
    \setlength{\unitlength}{420.70706898bp}%
    \ifx\svgscale\undefined%
      \relax%
    \else%
      \setlength{\unitlength}{\unitlength * \real{\svgscale}}%
    \fi%
  \else%
    \setlength{\unitlength}{\svgwidth}%
  \fi%
  \global\let\svgwidth\undefined%
  \global\let\svgscale\undefined%
  \makeatother%
  \begin{picture}(1,1.26357948)%
    \lineheight{1}%
    \setlength\tabcolsep{0pt}%
    \put(0,0){\includegraphics[width=\unitlength,page=1]{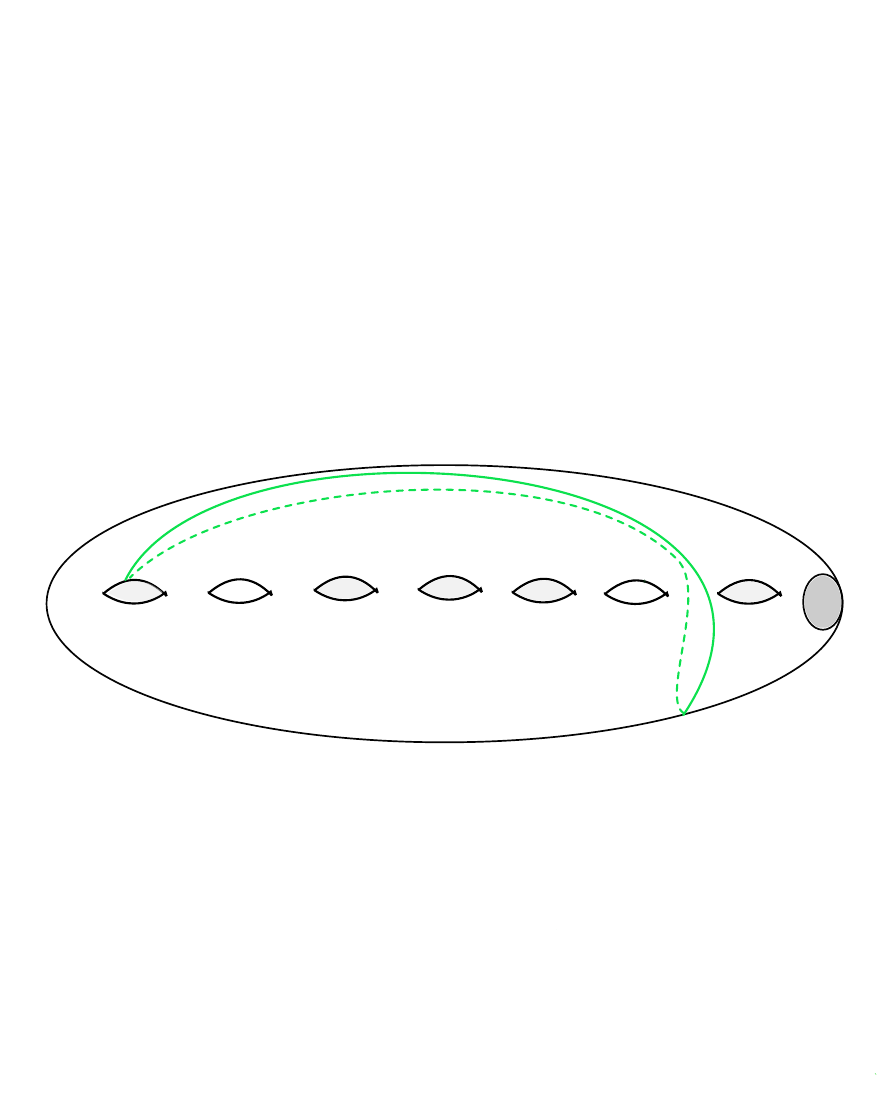}}%
    \put(0.44255412,0.51985539){\color[rgb]{0,0,0}\makebox(0,0)[lt]{\lineheight{1.25}\smash{\begin{tabular}[t]{l}$b_1'$\end{tabular}}}}%
    \put(0.46512605,0.69117543){\color[rgb]{0,0,0}\makebox(0,0)[lt]{\lineheight{1.25}\smash{\begin{tabular}[t]{l}$b_1$\end{tabular}}}}%
    \put(0,0){\includegraphics[width=\unitlength,page=2]{Isotopy_of_b2.pdf}}%
  \end{picture}%
\endgroup%

		\vspace{-4.4cm}
		\caption{}
		\label{fig8}
	\end{figure}
	
	\begin{figure}[!htb]
		\centering
		\def\svgwidth{11cm}
		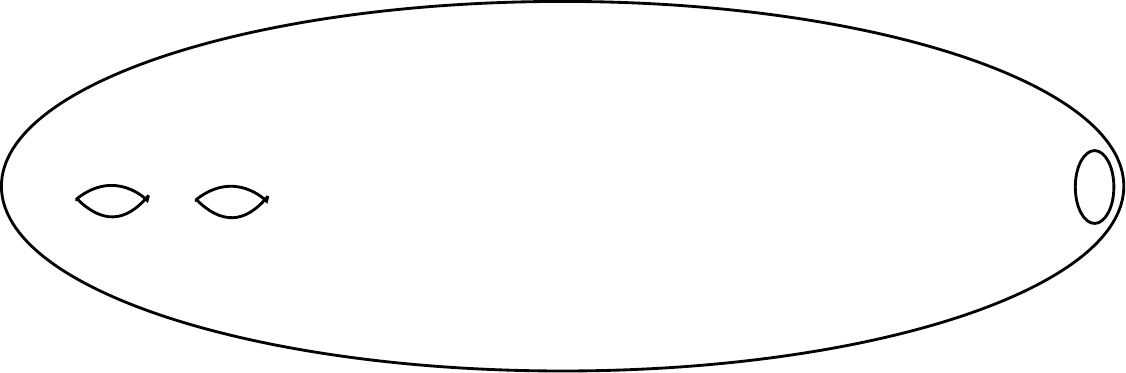
		\caption{}
		\label{fig9}
	\end{figure}
	
	\vspace{0.1cm}
	
	\item[(b)] Consider the embedding $h$ of  $\Sigma_{g,1}$ as shown in figure \ref{fig9}.

	As before, $h$ is described by embedding the chain of plumbed annuli such that the twists in each $A_i$ is positioned such that $d_{g-4}$ does not pass through the twisted regions in the $A_i$s. It is clear that the Dehn twists along the curves $c_1, c_2, c_3, c_4$ and $a_i$ ($1\leq i\leq g$) are $h$-extendible. The curve $d_1$ can be isotopic to the curve $d_1'$ (See figure \ref{fig9}. We see that $d_1'$ passes through $C_3$ once, $A_2$ twice, $C_2$ once, $A_1$ twice, and $C_1$ once. Therefore, $d_1'$ has surface framing $1$. Hence, $\tau_{d_1}$ is $h$-extendible.
	
	\noindent For $2 \leq j \leq g-4$, consider a curve $d_j'$ isotopic to $d_j$ on $\Sigma_{g,1}$ such that $d_j$ passes through $C_4$ once, $A_4$ twice, and continuing like this finally passes through $A_{g-j}$ twice and $C_{g-(j-1)}$ once. By adding up the framings, we see that $d_j'$ has surface framing $-1$. Therefore, $\tau_{d_j}$ ($2 \leq j \leq g-4$) is $h$-extendible.
	
	\noindent Similarly, we isotope $d_{g-3}$ to a curve $d_{g-3}'$ such that $d_{g-3}$ passes through $C_4$ once, $A_4$ twice, $C_5$ once, $A_5$ twice, and so on. Thus, $d_{g-3}'$ has surface framing : $-1$ for $g$ even, $1$ for $g$ odd. Hence, $\tau_{d_{g-3}}$ is $h$-extendible too. 
	 
	\begin{figure}[!htb]\label{}
		\centering
		\def\svgwidth{10cm}
\begingroup%
  \makeatletter%
  \providecommand\color[2][]{%
    \errmessage{(Inkscape) Color is used for the text in Inkscape, but the package 'color.sty' is not loaded}%
    \renewcommand\color[2][]{}%
  }%
  \providecommand\transparent[1]{%
    \errmessage{(Inkscape) Transparency is used (non-zero) for the text in Inkscape, but the package 'transparent.sty' is not loaded}%
    \renewcommand\transparent[1]{}%
  }%
  \providecommand\rotatebox[2]{#2}%
  \newcommand*\fsize{\dimexpr\f@size pt\relax}%
  \newcommand*\lineheight[1]{\fontsize{\fsize}{#1\fsize}\selectfont}%
  \ifx\svgwidth\undefined%
    \setlength{\unitlength}{405.13610248bp}%
    \ifx\svgscale\undefined%
      \relax%
    \else%
      \setlength{\unitlength}{\unitlength * \real{\svgscale}}%
    \fi%
  \else%
    \setlength{\unitlength}{\svgwidth}%
  \fi%
  \global\let\svgwidth\undefined%
  \global\let\svgscale\undefined%
  \makeatother%
  \begin{picture}(1,0.33097113)%
    \lineheight{1}%
    \setlength\tabcolsep{0pt}%
    \put(0,0){\includegraphics[width=\unitlength,page=1]{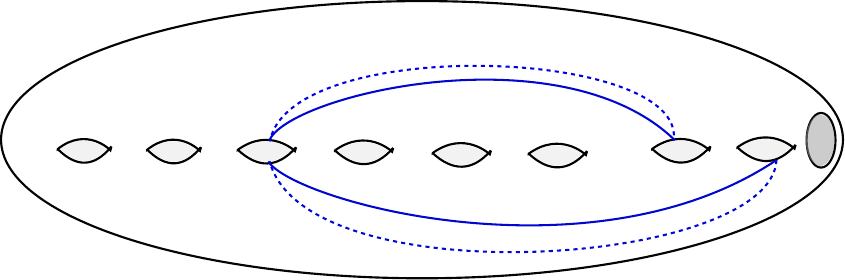}}%
    \put(0.70108725,0.14256989){\color[rgb]{0,0,0}\makebox(0,0)[lt]{\lineheight{1.25}\smash{\begin{tabular}[t]{l}$\cdots$\end{tabular}}}}%
    \put(0,0){\includegraphics[width=\unitlength,page=2]{Isotopy_of_di.pdf}}%
    \put(0.54983943,0.26167989){\color[rgb]{0,0,0}\makebox(0,0)[lt]{\lineheight{1.25}\smash{\begin{tabular}[t]{l}$d_2$\end{tabular}}}}%
    \put(0.5488145,0.19404998){\color[rgb]{0,0,0}\makebox(0,0)[lt]{\lineheight{1.25}\smash{\begin{tabular}[t]{l}$d_2'$\end{tabular}}}}%
    \put(0.59878371,0.0926764){\color[rgb]{0,0,0}\makebox(0,0)[lt]{\lineheight{1.25}\smash{\begin{tabular}[t]{l}$d_{g-3}'$\end{tabular}}}}%
    \put(0.52067569,0.02050887){\color[rgb]{0,0,0}\makebox(0,0)[lt]{\lineheight{1.25}\smash{\begin{tabular}[t]{l}$d_{g-3}$\end{tabular}}}}%
  \end{picture}%
\endgroup%

		\caption{}
		\label{fig10}
	\end{figure}

\end{enumerate}

\end{proof}		

\begin{exmp}
	Let $f:\Sigma_{g,1} \rightarrow \mathbb{D}^4$ be a Seifert type embedding and let $a_i$ ($1\leq i \leq g$), $b_j$ ($1\leq j \leq g-1$) and $c_k$ ($1\leq k \leq g$) be simple closed curves on $\Sigma_{g,q}$ (seefigure \ref{fig11}) such that $a_i$, $b_j$ and $c_k$ have surface framings $(-)^{i-1}$, $(-1)^j$ and $(-1)^{k-1}$ respectively. Then, we claim that the Dehn twists $\tau_{a_i}$ ($1 \leq i \leq g$), $\tau_{c_i}$ ($1\leq i \leq g-1$), $\tau_{b_{2l-1}}$ ($1 \leq 2l-1 \leq g$) are $f$-extendible.	
	
	\noindent By the Hopf annulus trick, $\tau_{a_i}$ ($1 \leq i \leq g$), $\tau_{c_k}$ ($1 \leq k \leq g-1$) and $\tau_{b_1}$ are $f$-extendible. The curve $b_{2k+1}$, when isotoped into a tubular neighborhood of $b_1 \cup a_1 \cup  a_2 \cup \dots \cup a_g \cup c_1 \cup c_2 \cup \dots \cup c_{g-1}$, would pass through the Hopf annulus neighbourhoods of the curve $c_{2k}, a_{2k}, a_{2k-1}, c_{2k-1},\cdots, c_1, a_1,$ and $b_1$. Since the twists in the Hopf annulus neighbourhoods of $a_i$ ($1 \leq i \leq g$) are chosen not to be on the side of $b_{2k+1}$ and the curves $c_i$s have alternating framing signs, $b_{2k+1}$s will have surface framing $-1$. Hence, $\tau_{b_{2k+1}}$s are $f$-extendible too.
	
\end{exmp}	

 \begin{figure}[!htb]\label{}
			\centering
			\def\svgwidth{12cm}
			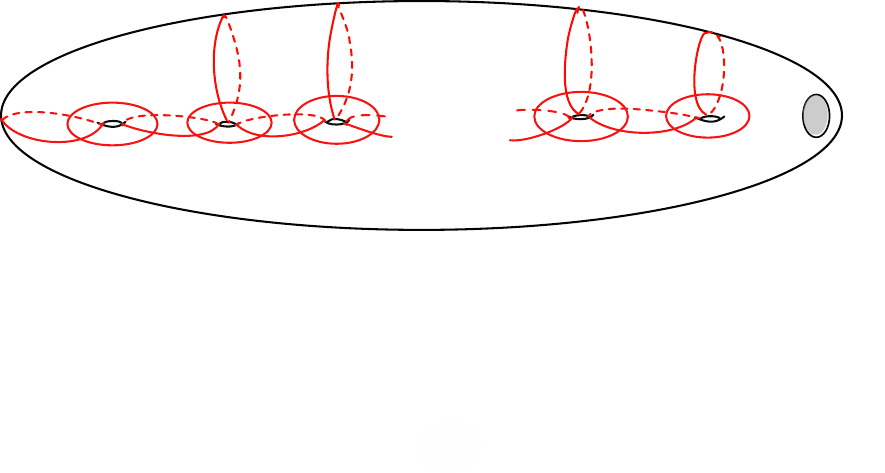
			\vspace{-3.5cm}
			\caption{ }
			\label{fig11}
\end{figure}

		\section{Flexible embeddings and sliceness of links}\label{sec5}
	
	Consider a system of simple closed curves on $\Sigma_{g,b}$, denoted by $\mathcal{H}$, as shown in figure \ref{humphreygen}. It is well known that Dehn twists along the curves in $\mathcal{H}$ generate $\mathcal{MCG}(\Sigma_{g,0})$. $\mathcal{H}$ is called the Humphreys generating set. Recall, that an $m$-component link $L$ in the boundary of a $4$-manifold $V$ is called \emph{slice}, if there exist $m$ properly embedded disjoint $2$-disks in $V$ with boundary $L$. As an application of Theorem \ref{thm1} we obtain the following result.

	\begin{theorem} \label{thm4}
		Let $W$ be a $4$-manifold with nonempty boundary such that $W$ is an integral homology $4$-ball and $\partial W$ is an integral homology $3$-sphere. Let $\mathcal{L}$ be an $m$-component link in $\partial W$ that bounds a proper embedding $f$ of $ \Sigma_{g,m}$ in $W$ such that for any curve $\alpha$ in $\mathcal{H}$, $\tau_{\alpha}$ is $f$-extendible. Then, $\mathcal{L}$ is not slice in any $4$-manifold $V$ bounding $\partial W$ such that $V$is an integral homology $4$-ball. 
	\end{theorem}
	
	\begin{proof}
		
		Let us assume that $\mathcal{L}$ in $\partial W$ is smoothly slice in some $4$-manifold $V$ bounding $\partial W$ such that $V$ is an integral homology $4$-ball. Let $\mathcal{D} = D_1 \sqcup D_2 \sqcup \cdots \sqcup D_m$ be the union of $m$ disjoint disks in $V$ bounding $\mathcal{L}$. Capping off $\Sigma_{g,m}$ by $\mathcal{D}$ gives an embedding $F$ of $\Sigma_{g,0}$ in $X = \frac{V \sqcup W}{\partial V \sim_{id} \partial W}$. Since $V$ and $W$ are integral homology $4$-balls with an integral homology $3$-sphere boundary, $X$ is an integral homology $4$-sphere. Thus, $F$ is characteristic. 
		
		\noindent Note that any element $\phi \in \mathcal{MCG}(\Sigma_{g,0})$ that is $f$-extendible, is also $F$-extendible. Thus, $\tau_\gamma$ is $F$-extendible for all $\gamma \in \mathcal{H}$. Since Dehn twists along the curves in $\mathcal{H}$ generate $\mathcal{MCG}(\Sigma_{g,0})$, $F$ is a flexible embedding. By Theorem \ref{thm1}, this gives a contradiction.
		
	\end{proof}
	
	\begin{exmp}
		
			Consider the case when $\partial W = \s^3$. One can construct examples of seifert surfaces of links satisfying the conditions in Theroem \ref{thm4}. Figure \ref{nl} gives two such examples of a $2$-componentlink. The picture on the left of figure \ref{nl} is an embedding of $\Sigma_{3,2}$, and the picture on the right is an embedding of $\Sigma_{2,2}$. Note that every core curve of an annulus (the red curves) in these two embeddings has surface framing $\pm1$. Therefore, Dehn twists along them are extendible. Also, for both the embeddings, the set of core curves is the Humphreys generating set. Thus, by Theorem \ref{thm4}, none of these $2$-links is slice in any integral homology $4$-ball bounding $\s^3$.   
		
		\begin{figure}[htbp] 
			
			\centering
			\def\svgwidth{14cm}
\begingroup%
  \makeatletter%
  \providecommand\color[2][]{%
    \errmessage{(Inkscape) Color is used for the text in Inkscape, but the package 'color.sty' is not loaded}%
    \renewcommand\color[2][]{}%
  }%
  \providecommand\transparent[1]{%
    \errmessage{(Inkscape) Transparency is used (non-zero) for the text in Inkscape, but the package 'transparent.sty' is not loaded}%
    \renewcommand\transparent[1]{}%
  }%
  \providecommand\rotatebox[2]{#2}%
  \newcommand*\fsize{\dimexpr\f@size pt\relax}%
  \newcommand*\lineheight[1]{\fontsize{\fsize}{#1\fsize}\selectfont}%
  \ifx\svgwidth\undefined%
    \setlength{\unitlength}{510.97460827bp}%
    \ifx\svgscale\undefined%
      \relax%
    \else%
      \setlength{\unitlength}{\unitlength * \real{\svgscale}}%
    \fi%
  \else%
    \setlength{\unitlength}{\svgwidth}%
  \fi%
  \global\let\svgwidth\undefined%
  \global\let\svgscale\undefined%
  \makeatother%
  \begin{picture}(1,0.36009056)%
    \lineheight{1}%
    \setlength\tabcolsep{0pt}%
    \put(0,0){\includegraphics[width=\unitlength,page=1]{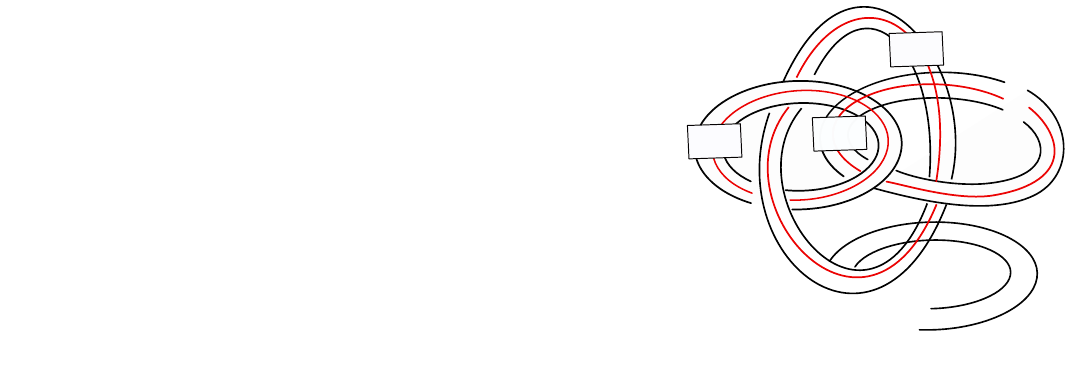}}%
    \put(0.64916802,0.21615959){\color[rgb]{0,0,0}\rotatebox{1.89175556}{\makebox(0,0)[lt]{\lineheight{1.25}\smash{\begin{tabular}[t]{l}$-1$\end{tabular}}}}}%
    \put(0.77860542,0.22530026){\color[rgb]{0,0,0}\rotatebox{1.89175627}{\makebox(0,0)[lt]{\lineheight{1.25}\smash{\begin{tabular}[t]{l}$1$\end{tabular}}}}}%
    \put(0.84076087,0.30554188){\color[rgb]{0,0,0}\rotatebox{1.89175556}{\makebox(0,0)[lt]{\lineheight{1.25}\smash{\begin{tabular}[t]{l}$-1$\end{tabular}}}}}%
    \put(0,0){\includegraphics[width=\unitlength,page=2]{non_slice_link.pdf}}%
    \put(0.87532403,0.0085937){\color[rgb]{0,0,0}\makebox(0,0)[lt]{\lineheight{1.25}\smash{\begin{tabular}[t]{l}$1$\end{tabular}}}}%
    \put(0.94628049,0.09338161){\color[rgb]{0,0,0}\makebox(0,0)[lt]{\lineheight{1.25}\smash{\begin{tabular}[t]{l}$-1$\end{tabular}}}}%
  \end{picture}%
\endgroup%

			\caption{}
			\label{nl}
			
		\end{figure}
		
	\end{exmp}

	\noindent There exist infinitely many (of different homotopy type) examples of a $4$-manifold that is an integral homology $4$-ball and has boundary $\s^3$. In particular, Ratcliffe and Tschantz \cite{RT} have constructed an infinite family of aspherical closed $4$-manifolds which are also integral homology $4$-spheres. Removing an open $4$-disk from each of those manifolds gives us an infinite family of integral homology $4$-balls with boundary $\s^3$.

	\section{Extendible mapping classes induced via Fibered Dehn twists on $\s^4$}\label{sec6}

	\subsection{A circle action on $\s^3$ via open book and $\textrm{T}_{\s^3}$} Consider the trivial open book of $\s^3$ given by $\textrm{OB}(\Sigma = \D^2, id)$. The open book induces an $\s^1$-action on $\s^3$ via the flow of a vector field $X$ on $\s^3$ whose time-$1$ map takes a page to itself and applies the identity map as monodromy on it. In a neighborhood of the binding, $(\partial \Sigma \times \D^2, (x,r,\theta))$, $X$ is given by $r \cdot \frac{\partial}{\partial \theta}$. In particular, the flow of $X$ fixes the binding unknot $U_0 = \partial \Sigma$ pointwise.  Let $p$ be a point in the interior of the collar neihghborhood $U_0 \times [1-\epsilon,1]$. Then the orbit of $p$ under this $\s^1$-action gives another unknot $U_1$ in $\s^3$. See the picture on the left of figure \ref{fig0}. By the definition of open book, $U_1$ naturally bounds a $2$-disk in $S^3$ that intersects the binding at a single point. Thus, we have an $\s^1$-action on $\s^3$ that fixes the binding $U_0$ and rotates $U_1$. Let $\textrm{T}_{\s^3}$ denote the corresponding fibered Dehn twist on $\s^3 \times [0,1]$.

	\begin{figure}[htbp] 
		
		\centering
		\def\svgwidth{11cm}
		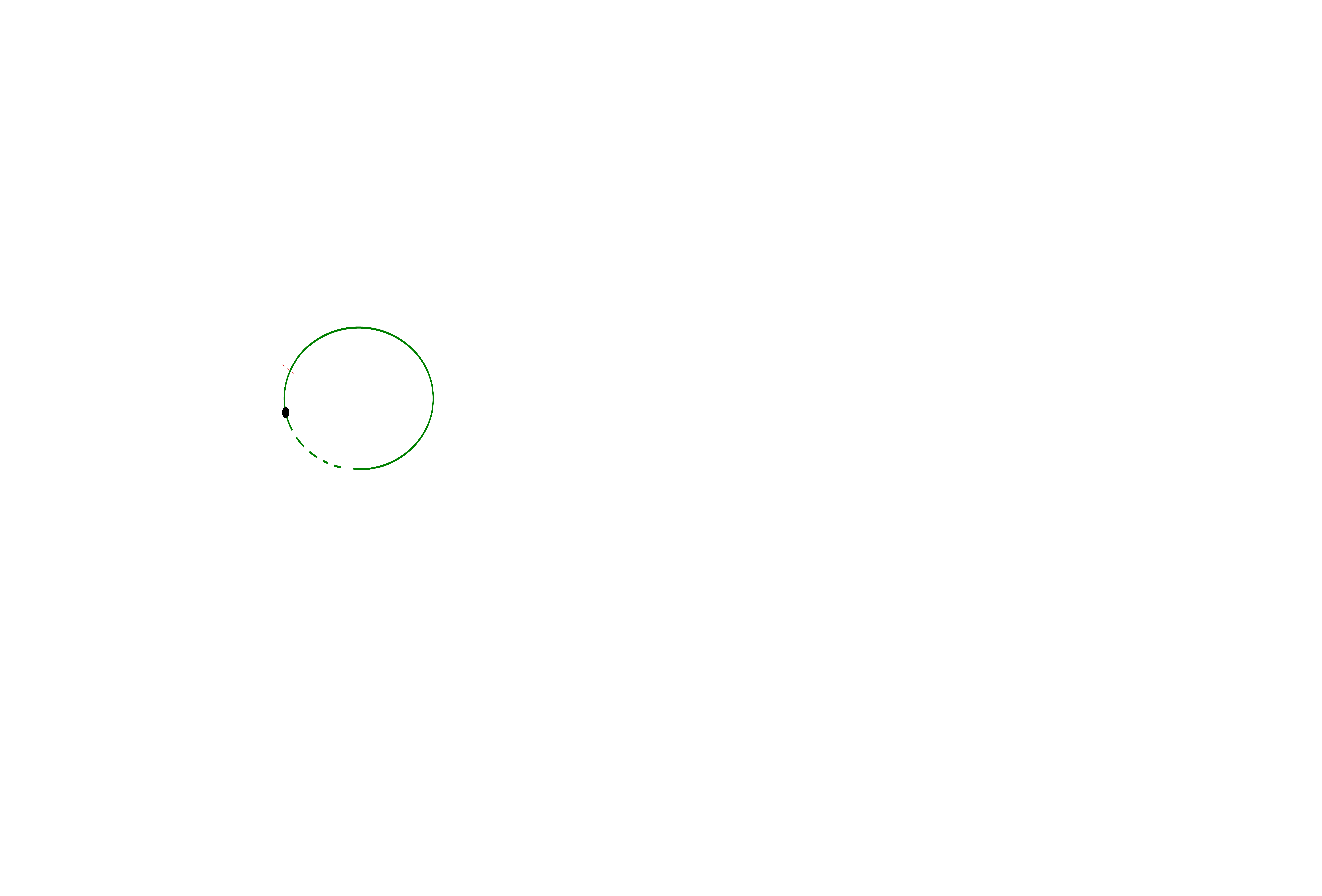
		\caption{}
		\label{fig0}
		
	\end{figure}

	\begin{exmp}

    To see how $\textrm{T}_{\s^3}$ induces a diffeomorphism on an embedded surface, consider an embedding of $\Sigma_{g-1}$ in $\s^4$ as in the picture on the right of figure \ref{fig0}. Here, we embed $\Sigma_{g-1}$ in a tubular neighborhood of an $\s^3 \subset \s^4$, diffeomorphic to $\s^3 \times [-1,2]$. We think of $\sigma_{g-1}$ as a union of three partss : two planar surfaces with $g$ boundary components $U_1, U_2, \dots, U_g$ and $g$ copies of cylinders $U_1 \times [0,1] \sqcup \dots \sqcup U_g \times [-1,2]$. We embed one of the planar parts in $\s^3 \times \{2\}$ (the upper shaded region on the right of figure \ref{fig0}) and the other one in $\s^3 \times \{-1\}$ (the lower shaded region on the left of figure \ref{fig0}). The middle part is embedded such that $U_i \times \{t\}$ embeds in $\s^3 \times \{t\}$ ($t \in [-1,2]$), for $1 \leq i \leq g$. We can arrange these embeddings so that for the trivial open book $\textrm{OB}(\D^2,id)$ of $\s^3$, the loops $U_i \times \{t\}$ become orbits of $g$ disjoint points in the interior of a page $D$ under the flow of the open book. Applying $\textrm{T}_{\s^3}$ with support $\s^3 \times [0,1]$ then induces the diffeomorphism $\tau_{U_1} \circ \tau_{U_2}\circ \dots \circ \tau_{U_g}$ on the embedded $\Sigma_{g-1}$.   		
		
			
			
		
    \end{exmp}
	
	\subsection{Generating the subgroup of extendible mapping classes}
	
	 Consider the subgroup of all extendible mapping classes for the trivial embedding $e_0 : \Sigma_g \rightarrow \s^4$. Let $SP_g[q_{st}]$ denote the subgroup of all $e_0$-extendible mapping classes in $\mathcal{MCG}(\Sigma_g)$. Hirose \cite{hirose1} proved that $SP_g[q_{st}]$ is isomorphic to the spin mapping class group $SP_g$ and gave a set of generators for $SP_g[q_{st}]$ as follows.
	
	$$X_i = \tau_{c_{i+1}} \circ \tau_{c_i} \circ \tau^{-1}_{c_{i+1}} \   \ (1 \leq i \leq 2g),$$ 
	$$ Y_{2j} = \tau_{c_{2j}} \circ \tau_{b_{2j}} \circ \tau^{-1}_{c_{2j}} \   \ (2 \leq j \leq g-1),$$
	$$ D_i = \tau^2_{c_i} \   \ (1 \leq i \leq 2g+1),$$
	$$ DB_{2j} = \tau^2_{b_{2j}} \   \ (2 \leq j \leq g-1),$$
	$$ Z_1 = \tau_{c_1} \circ \tau_{c_3} \circ \tau_{B_4}, Z_2 = \tau_{b_4} \circ \tau_{c_5} \circ \tau_{c_7}\circ \cdots \circ \tau_{c_{2g+1}}$$.
	
	\noindent Here, $c_i$s and $b_j$s are simple closed curves as shown in figure \ref{hg}.
	
	\begin{figure}[htbp] 
		
		\centering
		\def\svgwidth{12cm}
		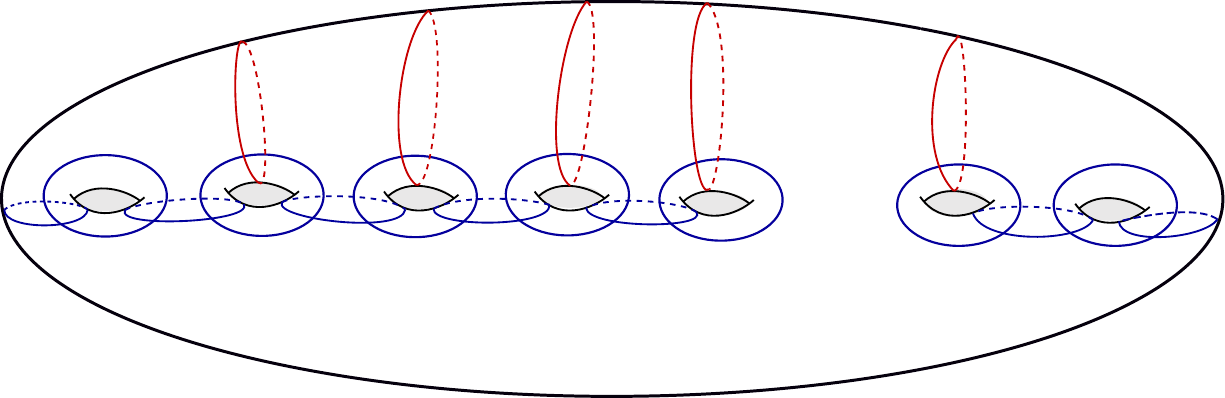
		\caption{}
		\label{hg}
		
	\end{figure}
	
	\vspace{0.1cm}
	
	\noindent We note that $X_i = \tau_{\tau_{c_{i+1}}(c_i)}$, and for the trivial embedding , the curve $\tau_{c_{i+1}}(c_i)$ is an unknot and has surface framing $\pm1$. Therefore, the $X_i$s can be induced by an isotopy of $\s^4$. Similarly, the $Y_{2j}$s can be induced by an isotopy of $\s^4$. The mapping classes $D_i$ and $DB_2j$ can also be induced via an isotopy of $\s^4$, by applying the tube trick.

	\begin{theorem} \label{thm5}
		For $g \geq2$, 
		
		\begin{enumerate}
			\item[(a)] $\textrm{T}_{\s^3}$ induces the mapping classes : $\beta_j=\tau_{b_{2j}} \circ \tau_{c_{2j+1}}\circ \tau_{b_{2j+1}} (2 \leq j \leq g-1), \beta_1= \tau_{c_1}\circ \tau_{c_3}\circ \tau_{b_4}, \beta_g=\tau_{b_{2g-2}}\circ \tau_{c_{2g-1}}\circ \tau_{c_{2g+1}}, \zeta=\tau_{b_4}\circ \tau_{c_5} \circ \tau_{c_7}\circ \cdots \circ \tau_{c_{2g+1}}$.
			
			\vspace{0.25cm}
			
			\item[(b)] $SP_g[q_{st}]$ is generated by the following mapping classes.  
			$$\alpha_i=\tau_{c_{2i+1}} \circ \tau_{c_{2i}} \circ \tau^{-1}_{c_{2i+1}} \ ( 1 \leq i \leq g \ ),$$
			$$\beta_j=\tau_{b_{2j}} \circ \tau_{c_{2j+1}}\circ \tau_{b_{2j+1}} \ ( 2 \leq j \leq g-1 \ ),$$
			$$ \beta_1= \tau_{c_1}\circ \tau_{c_3}\circ \tau_{b_4},~ \ \beta_g=\tau_{b_{2g-2}}\circ \tau_{c_{2g-1}}\circ \tau_{c_{2g+1}},$$
			$$ \gamma_1= \tau^2_{c_1}, \ \gamma_k= \tau^2_{b_{2k}} \ (2 \leq k \leq g-1),$$
			$$ \zeta=\tau_{b_4}\circ \tau_{c_5} \circ \tau_{c_7}\circ \cdots \circ \tau_{c_{2g+1}}.$$
		\end{enumerate}

	\end{theorem}

	\begin{proof} Let $G$ be a group generated by $\alpha_i$ ($1 \leq i \leq g$), $\beta_j$ ($1 \leq j \leq g$), $\gamma_k$ ($\leq k \leq g-1$),  and $\zeta$. 
			
			\begin{enumerate}
				
				\item[(a)]  By the Hopf annulus trick, the maps $\alpha_i$ ($1 \leq i \leq g$) can be induced by isotopies of $\s^4$. Similarly, by the tube trick, the maps $\gamma_1$ and $\gamma_k$ ($2 \leq k \leq g-1$) can be induced by an isotopy of $\s^4$. To show that the remaining geneartors of $G$ can be induced by $\textrm{T}_{\s^3}$, we describe an embedding isotopic to the trivial embedding, for each of those generators. We take a tubular neighborhood of the standard $\s^3 \subset \s^4$, diffeomorphic to $\s^3 \times [-1, 2]$.

				\begin{figure}[htbp] 
					
					\centering
					\vspace{-4.5cm}
					\def\svgwidth{13cm}
					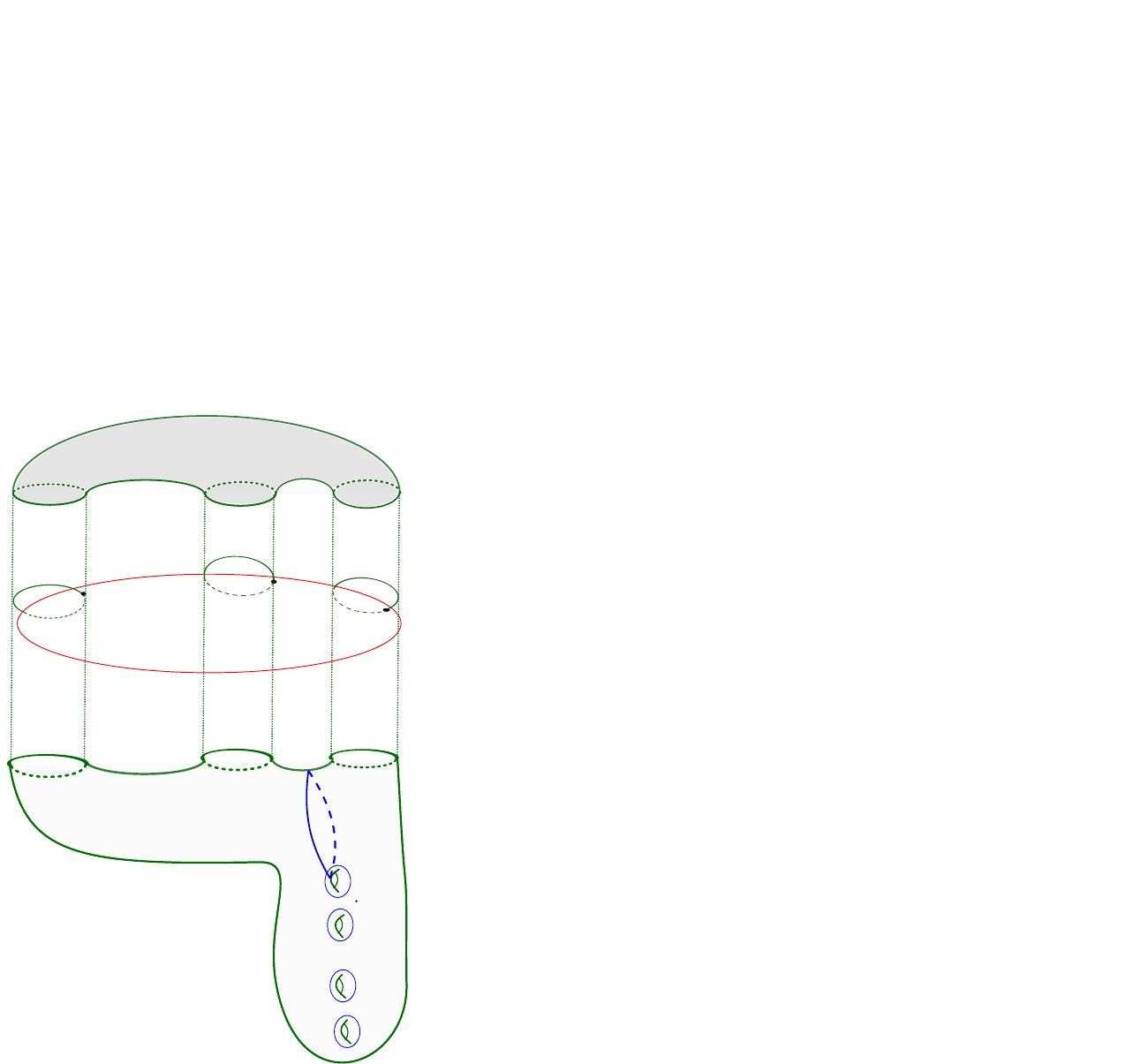
					\caption{}
					\label{figbeta1g}
					
				\end{figure}

				\noindent To induce the map $\beta_1$, we isotope the trivial embedding of $\Sigma_g$ to an embedding as shown on the left of figure \ref{figbeta1g}. Here, the upper shaded region of the surface is contained in the level $\s^3 \times \{-1\}$ and the lower shaded region is contained in the level $\s^3 \times \{2\}$. The middle portion is given by $c_1 \times [-1,2] \sqcup c_3 \times [-1,2] \sqcup b_4 \times [-1,2]$ in $\s^3 \times [-1,2]$. Now consider a trivial open book on $\s^3 \times \{t\}$ ($t \in [0,1]$) such that the curves $c_1 \times \{t\}, c_3 \times \{t\}$ and $b_4 \times \{t\}$ are orbits of three disjoint points in the interior of a $2$-disk page $D$, under the flow of the open book. The binding $\partial D$ is denoted by the curve. Clearly, applying $\textrm{T}_{\s^3}$ (supported on $\s^3 \times [0,1]$) then induces $\beta_1$ on the embedded surface. A similar embedding for inducing $\beta_g$ is shown on the left of figure \ref{figbeta1g}.

				\begin{figure}[htbp] 
					
					\centering
					\def\svgwidth{12cm}
					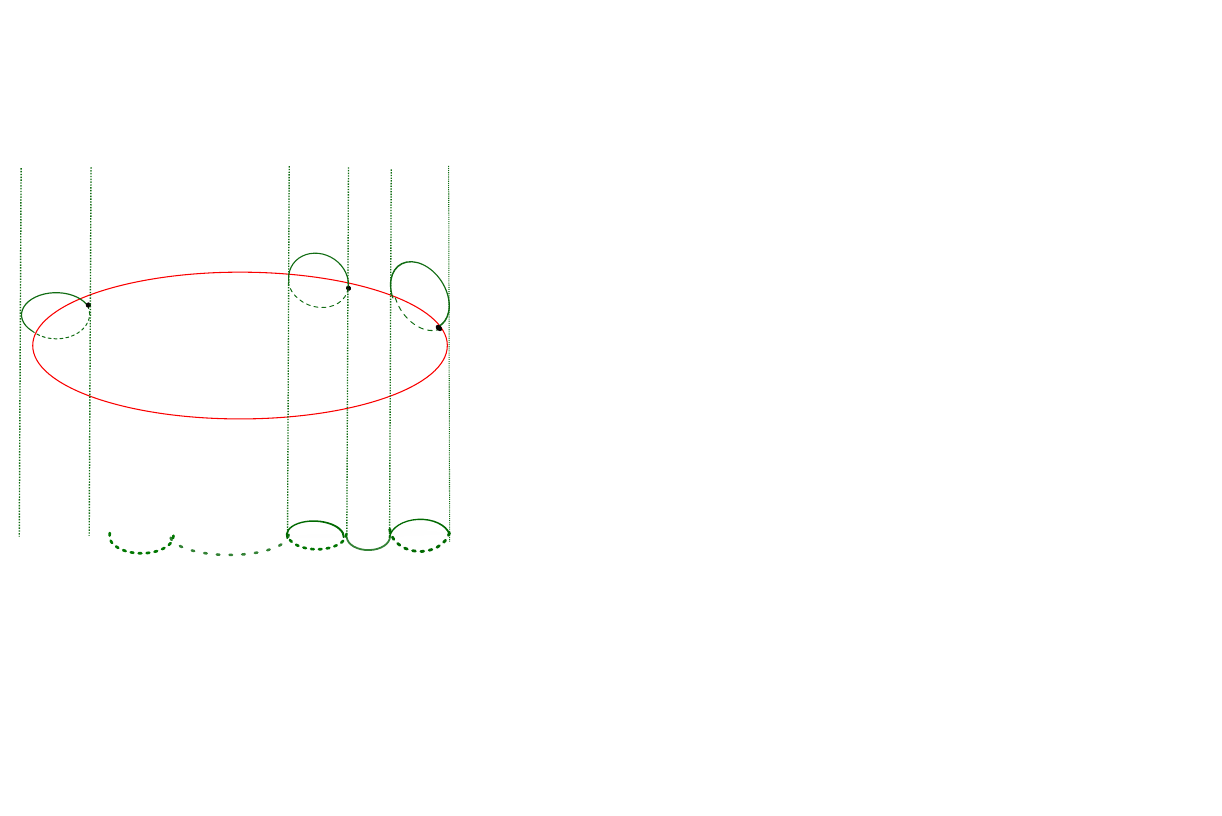
					\caption{}
					\label{figzj}
					
				\end{figure}
				
				\noindent To induce the map $\zeta$, we isotope the trivial embedding of $\Sigma_g$ to an embedding as shown on the left of figure \ref{figzj}. In a similar fashion, the upper (lower) shaded region of $\Sigma_g$ is contained in the level $\s^3 \times \{-1\}$ ($\s^3 \times \{2\}$), and the middle part consists of disjoint union of cylinders : $(b_4 \sqcup c_5 \sqcup c_7 \sqcup \dots \sqcup c_{2g-1} \sqcup c_{2g+1}) \times [-2,1]$. We again arrange these curves so that they become orbits of $g-1$ disjoint points in the interior of a $2$-disk page $D$, under the flow of the open book. Therefore, applying $\textrm{T}_{\s^3}$ will induce $\zeta$ on the embedded surface.

				\noindent Finally, to induces $\beta_j$ ($2 \leq j \leq g-1$), we isotope the trivial embedding of $\Sigma_g$ to an embedding as shown on the right of figure \ref{figzj}
				
				\vspace{0.1cm}
				
				\item[(b)] By statement $(a)$ above, $G \subset SP_g[q_{st}]$. Thus, we need to show that $SP_g[q_{st}] \subset G$. It is enough to show that $ X_i, Y_{2j}, D_i, DB_{2j} \in G$. Note that $X_{2i+1}=\alpha_i \ (1\leq \i \leq g )$, $DB_{2k}=\gamma_k(2\leq g-1)$ and $\tau^2_{c_1}=\gamma_1 $ are already elements in $G$. 
				
				\begin{align*}\mathrm{For ~2 \leq i \leq g-1,~}
					D_{2i+1} & = \tau^2_{c_{2i+1}}\\
					& = \tau^2_{c_{2i+1}} \circ (\tau_{b_{2i}}\circ\tau_{b_{2i+2}} \circ \tau^{-1}_{b_{2i+2}}\circ \tau^{-1}_{b_{2i}})^2\\
					& = (\tau_{b_{2i}}\circ\tau_{c_{2i+1}} \circ \tau_{b_{2i+2}})^2\circ \tau^{-2}_{b_{2i}}\circ \tau^{-2}_{b_{2i+2}} \in G.
				\end{align*}
				
				For $i=1,$	$D_3=\tau^2_{c_3}=(\tau_{c_1} \circ \tau_{c_3} \circ \tau_{b_4})^2 \circ \tau^{-2}_{c_1}\circ \tau^{-2}_{b_4} \in G.$ 
				\begin{align*}
					\mathrm{For ~i=g,~}
					D_{2g+1} & = \tau^2_{c_{2g+1}}\\
					& = \tau^2_{c_{2g+1}}\circ (\tau_{b_{2g-2}}\circ \tau_{c_{2g-1}})^2 \circ (\tau_{b_{2g-2}}\circ \tau_{c_{2g-1}})^{-2}\\
					& = (\tau_{b_{2g-2}}\circ \tau_{c_{2g-1}}\circ \tau_{c_{2g+1}})^2\circ \tau^{-2}_{c_{2g-1}}\circ \tau^{-2}_{b_{2g-2}} \in G.
				\end{align*}

				\begin{align*} \mathrm{For ~2 \leq j \leq g-1,~}
					Y_{2j} & = \tau_{c_{2j}} \circ \tau_{b_{2j}} \circ \tau^{-1}_{c_{2j}}\\
					\begin{split}
						& = (\tau_{c_{2j+1}} \circ \tau_{b_{2j}}\circ \tau_{b_{2j+2}})^{-1} \circ (\tau_{c_{2j+1}} \circ \tau_{b_{2j}}\circ \tau_{b_{2j+2}}) \circ (\tau_{c_{2j}} \circ \tau_{b_{2j}} \circ \tau^{-1}_{c_{2j}}) \circ\\
						& \qquad (\tau_{c_{2j+1}} \circ \tau_{b_{2j}}\circ \tau_{b_{2j+2}})^{-1} \circ (\tau_{c_{2j+1} }\circ \tau_{b_{2j}}\circ \tau_{b_{2j+2}})\\
						& = (\tau_{c_{2j+1}} \circ \tau_{b_{2j}}\circ \tau_{b_{2j+2}})^{-1} \circ \tau_{b_{2j+2}}\circ \tau_{c_{2j+1}} \circ (\tau_{b_{2j}}\circ \tau_{c_{2j}} \circ \tau_{b_{2j}} \circ \tau^{-1}_{c_{2j}} \circ\tau^{-1}_{b_{2j}})\\
						& \qquad \circ \tau^{-1}_{c_{2j+1}} \circ \tau^{-1}_{b_{2j+2}}(\tau_{c_{2j+1} }\circ \tau_{b_{2j}}\circ \tau_{b_{2j+2}})\\
						& = (\tau_{c_{2j+1}} \circ \tau_{b_{2j}}\circ \tau_{b_{2j+2}})^{-1} \circ \tau_{b_{2j+2}}\circ \tau_{c_{2j+1}} \circ (\tau_{c_{2j}}) \circ \tau^{-1}_{c_{2j+1}} \circ  \tau^{-1}_{b_{2j+2}}\circ \\
						& \quad (\tau_{c_{2j+1} }\circ \tau_{b_{2j}}\circ \tau_{b_{2j+2}})\\
						& = (\tau_{c_{2j+1}} \circ \tau_{b_{2j}}\circ \tau_{b_{2j+2}})^{-1}\circ (\tau_{c_{2j+1}}\circ \tau_{c_{2j}}\circ \tau^{-1}_{c_{2j+1}})\circ(\tau_{c_{2j+1} }\circ \tau_{b_{2j}}\circ \tau_{b_{2j+2}})\in G\\
					\end{split}
				\end{align*}	
				\noindent Since $b_{2i}$ and $c_{2i}$ intersect only at one point,
				\begin{equation}\label{eq1}
					\tau_{b_{2i}}=\tau_{c_{2i}} \circ \tau_{b_{2i}}\circ \tau_{c_{2i}}\circ \tau^{-1}_{b_{2i}}\circ \tau^{-1}_{c_{2i}}.
				\end{equation}
				(\ref{eq1}) gives $\tau_{b_{2i}}\circ \tau_{c_{2i}}=\tau^{-1}_{c_{2i}} \circ \tau_{b_{2i}}\circ \tau_{c_{2i}}\circ \tau_{b_{2i}} (2\leq i \leq g-1).$
				\noindent Now 
				\begin{align*}
					\tau^{-1}_{c_{2i}}\circ \tau_{b_{2i}}\circ \tau_{c_{2i}} & = \tau^{-1}_{c_{2i}}\circ\tau^{-1}_{c_{2i}} \circ \tau_{b_{2i}}\circ \tau_{c_{2i}}\circ \tau_{b_{2i}}\\
					& =\tau^{-2}_{c_{2i}} \circ (\tau_{b_{2i}}\circ \tau_{c_{2i}}\circ \tau_{b_{2i}})\\
					& =\tau^{-2}_{c_{2i}} \circ (\tau_{b_{2i}}\circ \tau_{c_{2i}}\circ \tau^{-1}_{b_{2i}}) \circ \tau^2_{b_{2i}}.\\
				\end{align*}
				\noindent Thus for $(2 \leq i \leq g-1),$ $\tau^2_{c_{2i}}= (\tau_{b_{2i}}\circ \tau_{c_{2i}}\circ \tau^{-1}_{b_{2i}}) \circ \tau^2_{b_{2i}}\circ (\tau^{-1}_{c_{2i}}\circ \tau^{-1}_{b_{2i}}\circ \tau_{c_{2i}})$. 
				From (\ref{eq1}), $\tau^{-1}_{c_{2i}}\circ \tau^{-1}_{b_{2i}}\circ \tau_{c_{2i}}= \tau_{b_{2i}}\circ \tau^{-1}_{c_{2i}}\circ \tau^{-1}_{b_{2i}}=(\tau_{b_{2i}}\circ \tau_{c_{2i}}\circ \tau^{-1}_{b_{2i}})^{-1}.$ Therefore, to show that $D_{2i}= \tau^2_{c_{2i}} \in G$ it is enough to show that $\tau_{b_{2i}}\circ \tau_{c_{2i}}\circ \tau^{-1}_{b_{2i}} \in G$. 
				
				\begin{align*}
					\tau_{b_{2i}}\circ \tau_{c_{2i}}\circ \tau^{-1}_{b_{2i}} & = \tau_{b_{2i}}\circ \tau_{c_{2i}}\circ \tau^{-1}_{b_{2i}}\\
					& = \tau_{b_{2i}}\circ (\tau_{b_{2i+2}} \circ \tau^{-1}_{c_{2i+1}}\circ\tau_{c_{2i+1}} )\circ\tau_{c_{2i}}\circ(\tau^{-1}_{c_{2i+1}}\circ\tau_{c_{2i+1}} \circ\tau^{-1}_{b_{2i+2}})\circ \tau^{-1}_{b_{2i}}\\
					& = \tau_{b_{2i}}\circ \tau_{b_{2i+2}} \circ \tau^{-1}_{c_{2i+1}}\circ(\tau_{c_{2i+1}} \circ\tau_{c_{2i}}\circ\tau^{-1}_{c_{2i+1}})\circ\tau_{c_{2i+1}} \circ\tau^{-1}_{b_{2i+2}}\circ \tau^{-1}_{b_{2i}}\\
					\begin{split}
					& = (\tau_{b_{2i}}\circ \tau_{b_{2i+2}} \circ\tau_{c_{2i+1}})\circ \tau^{-2}_{c_{2i+1}}\circ(\tau_{c_{2i+1}} \circ\tau_{c_{2i}}\circ\tau^{-1}_{c_{2i+1}})\circ\tau^2_{c_{2i+1}} \circ\\
					& \quad (\tau^{-1}_{c_{2i+1}}\circ\tau^{-1}_{b_{2i+2}}\circ \tau^{-1}_{b_{2i}})\in G
					\end{split}
					\end{align*}
	\noindent Hence, for $(2 \leq i \leq g-1),~ D_{2i} \in G.$ 
				\begin{align*}
					\mathrm{For ~i=1,~}
					D_2 & = \tau^2_{c_2}\\
					& = \tau_{c_2} \circ (\tau_{c_1}\circ \tau_{c_2} \circ \tau^{-1}_{c_2} \circ \tau^{-1}_{c_1})\circ \tau_{c_2}\\
					& = (\tau_{c_2} \circ \tau_{c_1}\circ \tau_{c_2}) \circ (\tau^{-1}_{c_2} \circ \tau^{-1}_{c_1})\circ \tau_{c_2})\\
					& = (\tau_{c_1} \circ \tau_{c_2}\circ \tau_{c_1}) \circ (\tau^{-1}_{c_2} \circ \tau^{-1}_{c_1})\circ \tau_{c_2})\\
					& = (\tau_{c_1} \circ \tau_{c_2}\circ \tau^{-1}_{c_1}) \circ \tau^2_{c_1}\circ (\tau^{-1}_{c_2} \circ \tau^{-1}_{c_1}\circ \tau_{c_2})\\
					& = (\tau_{c_1} \circ \tau_{c_2}\circ \tau^{-1}_{c_1}) \circ \tau^2_{c_1}\circ (\tau_{c_1} \circ \tau_{c_2}\circ \tau^{-1}_{c_1})^{-1}\\
				\end{align*}
				
				\noindent Since, $\tau_{c_1} \circ \tau_{c_2}\circ \tau^{-1}_{c_1}=(\tau_{c_1} \circ \tau_{c_3} \circ \tau_{b_4}) \circ \tau^{-2}_{c_3}\circ (\tau_{c_3} \circ \tau_{c_2}\circ \tau^{-1}_{c_3})\circ (\tau^{-1}_{c_1}\circ \tau^{-1}_{c_3} \circ \tau^{-1}_{b_4}) \circ \tau^2_{c_3} \in G$, $D_2$ is in $G.$
				
				\begin{align*}
					\mathrm{For~ i=g,~}
					D_{2g} & = \tau^2_{c_{2g}}\\
					& =  \tau^2_{c_{2g}}\circ (\tau_{c_{2g+1}} \circ \tau_{c_{2g}})\circ (\tau_{c_{2g+1}} \circ \tau_{c_{2g}})^{-1}\\
					& =  \tau_{c_{2g}}\circ (\tau_{c_{2g}}\circ \tau_{c_{2g+1}} \circ \tau_{c_{2g}})\circ (\tau_{c_{2g+1}} \circ \tau_{c_{2g}})^{-1}\\
					& =  \tau_{c_{2g}}\circ (\tau_{c_{2g+1}}\circ \tau_{c_{2g}} \circ \tau_{c_{2g+1}})\circ (\tau_{c_{2g+1}} \circ \tau_{c_{2g}})^{-1}\\
					& =  \tau_{c_{2g}}\circ (\tau_{c_{2g+1}}\circ \tau_{c_{2g}} \circ \tau_{c_{2g+1}})\circ (\tau_{c_{2g+1}} \circ \tau_{c_{2g}})^{-1}\\
					& =  \tau_{c_{2g}}\circ \tau_{c_{2g+1}}\circ \tau_{c_{2g}} \circ \tau_{c_{2g+1}}\circ \tau^{-1}_{c_{2g}} \circ \tau^{-1}_{c_{2g+1}}\\
					& =  (\tau_{c_{2g+1}}\circ \tau_{c_{2g}}\circ \tau_{c_{2g+1}}) \circ \tau_{c_{2g+1}}\circ \tau^{-1}_{c_{2g}} \circ \tau^{-1}_{c_{2g+1}}\\
					& =  (\tau_{c_{2g+1}}\circ \tau_{c_{2g}}\circ \tau^{-1}_{c_{2g+1}}) \circ\tau^2_{c_{2g+1}}  \circ (\tau_{c_{2g+1}}\circ \tau^{-1}_{c_{2g}} \circ \tau^{-1}_{c_{2g+1}}) \in G\\
				\end{align*}
				
				\begin{align*}
					\tau_{c_{2i}}\circ \tau_{c_{2i-1}} \circ \tau^{-1}_{c_{2i}} & =\tau_{b_{2i}} \circ \tau_{c_{2i}}\circ \tau_{b_{2i}}\circ \tau^{-1}_{c_{2i}}\circ \tau^{-1}_{b_{2i}}\circ \tau_{c_{2i-1}} \circ \tau^{-1}_{c_{2i}}\\ 
					& =(\tau_{b_{2i}} \circ \tau_{c_{2i}}\circ \tau_{b_{2i}})\circ \tau^{-1}_{c_{2i}}\circ \tau_{c_{2i-1}} \circ \tau^{-1}_{b_{2i}}\circ \tau^{-1}_{c_{2i}}\\
					& =(\tau_{c_{2i}} \circ \tau_{b_{2i}}\circ \tau_{c_{2i}})\circ \tau^{-1}_{c_{2i}}\circ \tau_{c_{2i-1}} \circ \tau^{-1}_{b_{2i}}\circ \tau^{-1}_{c_{2i}}\\
					& =(\tau_{c_{2i}} \circ \tau_{b_{2i}}\circ \tau^{-1}_{c_{2i}})\circ \tau_{c_{2i}}\circ \tau_{c_{2i-1}} \circ \tau^{-1}_{b_{2i}}\circ \tau^{-1}_{c_{2i}}\\
					& =(\tau_{c_{2i}} \circ \tau_{b_{2i}}\circ \tau^{-1}_{c_{2i}})\circ \tau_{c_{2i}}\circ (\tau_{c_{2i}} \circ \tau_{c_{2i-1}} \circ  \tau_{c_{2i}} \circ  \tau^{-1}_{c_{2i-1}}\circ \tau^{-1}_{c_{2i}} )\circ \tau^{-1}_{b_{2i}}\circ \tau^{-1}_{c_{2i}}\\
					& =(\tau_{c_{2i}} \circ \tau_{b_{2i}}\circ \tau^{-1}_{c_{2i}})\circ \tau^2_{c_{2i}}\circ (\tau_{c_{2i-1}} \circ  \tau_{c_{2i}} \circ  \tau^{-1}_{c_{2i-1}})\circ (\tau^{-1}_{c_{2i}} \circ \tau^{-1}_{b_{2i}}\circ \tau^{-1}_{c_{2i}})\\
					& =(\tau_{c_{2i}} \circ \tau_{b_{2i}}\circ \tau^{-1}_{c_{2i}})\circ \tau^2_{c_{2i}}\circ (\tau_{c_{2i-1}} \circ  \tau_{c_{2i}} \circ  \tau^{-1}_{c_{2i-1}})\circ \tau^{-2}_{c_{2i}} \circ(\tau_{c_{2i}} \tau^{-1}_{b_{2i}}\circ \tau^{-1}_{c_{2i}})\\
				\end{align*}
				\begin{align*}
					\tau_{c_{2i-1}} \circ  \tau_{c_{2i}} \circ  \tau^{-1}_{c_{2i-1}} & = \tau_{c_{2i-1}} \circ (\tau_{b_{2i-2}}\circ\tau^{-1}_{b_{2i}} \circ \tau_{b_{2i}} ) \circ \tau_{c_{2i}} \circ (\tau^{-1}_{b_{2i}} \circ \tau_{b_{2i}} \circ\tau^{-1}_{b_{2i-2}})\circ \tau^{-1}_{c_{2i-1}} \\
					\begin{split}
					& = (\tau_{c_{2i-1}} \circ \tau_{b_{2i-2}}\circ \tau_{b_{2i}})\circ\tau^{-2}_{b_{2i}} \circ( \tau_{b_{2i}}  \circ \tau_{c_{2i}} \circ \tau^{-1}_{b_{2i}}) \circ \tau^2_{b_{2i}} \circ\\
					& \quad (\tau^{-1}_{b_{2i-2}}\circ \tau^{-1}_{c_{2i-1}}\circ \tau^{-1}_{b_{2i}}) \\
					& = (\tau_{c_{2i-1}} \circ \tau_{b_{2i-2}}\circ \tau_{b_{2i}})\circ\tau^{-2}_{b_{2i}} \circ( \tau_{b_{2i}}  \circ \tau_{c_{2i}} \circ \tau^{-1}_{b_{2i}}) \circ \tau^2_{b_{2i}} \circ\\
					& \quad (\tau_{c_{2i-1}} \circ\tau_{b_{2i-2}}\circ \tau_{b_{2i}})^{-1} \in G
					\end{split}
					\end{align*}

			\end{enumerate}

\end{proof}

	\noindent As mentioned in section \ref{sec5}, Hamenst\"{a}dt \cite{ursula} has given a set of $2g+1$ generators for the group $SP_g$. It will be interesting to know an answer to the following question.
	
	\begin{question}
		Can one construct an embedding $f : \Sigma_g \rightarrow \s^4$ such that the subgroup of $f$-extendible mapping classes $SP_g[q_f]$ is isomorphic to $SP_g$, and $SP_g[q_f]$ is generated by a set of $(2g+1)$ $f$-extendible elements?
	\end{question}

\end{document}